\documentclass[10pt]{article}

\usepackage{amsfonts}
\usepackage{marginnote}
\usepackage{amsmath,amsthm,amscd,amssymb,mathrsfs,setspace}
\usepackage{latexsym,epsf,epsfig}
\usepackage{cite}
\usepackage{color}
\usepackage[hmargin=1.2in,vmargin=1in]{geometry}
\usepackage[T1]{fontenc}
\usepackage[utf8]{inputenc}
\usepackage{combelow}
\usepackage{newunicodechar}
\usepackage{tikz}
\usepackage{mathtools}
\usepackage{tabularx}
\usepackage{blindtext}
\usepackage{hyperref}
\usepackage{soul}

  \chardef\forshowkeys=0
  \chardef\refcheck=0
  \chardef\showllabel=0
  \chardef\sketches=0

\ifnum\forshowkeys=1
  
  \usepackage[notref,notcite,color]{showkeys}
\fi

\ifnum\refcheck=1
  \usepackage{refcheck}
\fi

\newcommand{\R}{\mathbb{R}}

\newcommand{\Z}{ \mathbb{Z}}

\theoremstyle{plain}
\newtheorem{Theorem}{Theorem}[section]

\newtheorem{Lemma}[Theorem]{Lemma}
\newtheorem{Remark}[Theorem]{Remark}

\newtheorem{assumption}{Assumption}
\numberwithin{equation}{section}
\numberwithin{Theorem}{section}
\def\lec{\lesssim}

\ifnum\showllabel=1
 \def\llabel#1{\marginnote{\color{lightgray}\rm\small(#1)}[-0.0cm]\notag}
\else
\def\llabel#1{\notag}
\fi

\numberwithin{equation}{section}
\definecolor{mygray}{rgb}{.6,.6,.6}
\definecolor{myblue}{rgb}{9, 0, 1}
\definecolor{colorforkeys}{rgb}{1.0,0.0,0.0}
\newlength\mytemplen
\newsavebox\mytempbox
\makeatletter
\newcommand\mybluebox{%
\@ifnextchar[
{\@mybluebox}%
{\@mybluebox[0pt]}}
\def\@mybluebox[#1]{%
\@ifnextchar[
{\@@mybluebox[#1]}%
{\@@mybluebox[#1][0pt]}}
\def\@@mybluebox[#1][#2]#3{
\sbox\mytempbox{#3}%
\mytemplen\ht\mytempbox
\advance\mytemplen #1\relax
\ht\mytempbox\mytemplen
\mytemplen\dp\mytempbox
\advance\mytemplen #2\relax
\dp\mytempbox\mytemplen
\colorbox{myblue}{\hspace{1em}\usebox{\mytempbox}\hspace{1em}}}
\makeatother
\def\am{a_{\text{m}}}
\def\ab{a_{\text{b}}}

\def\eold{\colb }
\def\rr{r}

\def\qqq{u}

\def\inon#1{\hbox{\quad{}}\hbox{#1}}                
\def\onon#1{\inon{on~$#1$}}

\def\andand{\text{\indeq and\indeq}}

\def\Tr{\mathop{\rm Tr}\nolimits}    
\def\div{\mathop{\rm div}\nolimits}
\def\curl{\mathop{\rm curl}\nolimits}
\def\Sym{\mathop{\rm Sym}\nolimits}
  
\def\supp{\mathop{\rm supp}\nolimits}
\def\indeq{\quad{}}

\def\colb{\color{black}}

\definecolor{colorgggg}{rgb}{0.1,0.5,0.3}
\definecolor{colorllll}{rgb}{0.0,0.7,0.0}
\definecolor{colorhhhh}{rgb}{0.3,0.75,0.4}
\definecolor{colorpppp}{rgb}{0.7,0.0,0.2}
\definecolor{coloroooo}{rgb}{0.45,0.0,0.0}
\definecolor{colorqqqq}{rgb}{0.1,0.7,0}

\def\cole{\color{coloroooo}}
\def\cole{{}}

\definecolor{coloraaaa}{rgb}{0.6,0.6,0.6}

\def\comma{ {\rm ,\qquad{}} }            
          
\def\nts#1{{\color{blue}\hbox{\bf ~#1~}}}

\def\fractext#1#2{{#1}/{#2}}

 \def\colb{\color{black}}

\def\comma{ {\rm ,\qquad{}} }        

\begin{document}
\baselineskip=5.8truemm
\title{Inviscid fluid interacting with a nonlinear two-dimensional plate}

\author{\small
\begin{tabular}[t]{c@{\extracolsep{1em}}cc@{\extracolsep{1em}}ccc@{\extracolsep{1em}}}
 Abhishek Balakrishna & Igor Kukavica & Boris Muha & Amjad Tuffaha \\ 
\it USC & \it USC& \it UNiZG & \it AUS\\
\it Los Angeles, CA & \it Los Angeles, CA & \it Zagreb, Croatia &\it Sharjah, UAE\\  
ab45315@usc.edu & kukavica@usc.edu & borism@math.hr  & atufaha\char'100aus.edu \\
\end{tabular}
}

\maketitle

\bf Abstract\rm: We address a moving boundary problem that consists of a system of equations
modeling an inviscid fluid interacting with a two-dimensional nonlinear Koiter
plate at the boundary. We derive a priori estimates needed to prove the local-in-time existence of solutions. We use the Arbitrary Lagrange Euler (ALE) coordinates to fix the domain and obtain careful estimates for the nonlinear Koiter plate, ALE velocity, and pressure {without any viscoelastic smoothing}. For the nonlinear Koiter plate, higher order energy estimates are obtained, whereas estimates for the ALE pressure are obtained by setting up an elliptic problem. For the ALE velocity, the bounds are obtained through div-curl estimates by estimating the ALE vorticity. We then extend our results in two directions: (1) to include fractional Sobolev spaces and (2) to incorporate the normalized second fundamental form.

\section{Introduction}
In this paper, we consider the inviscid free boundary problem that couples the incompressible Euler equation with the nonlinear Koiter plate. The system consists of a fluid channel with a free elastic boundary on top, modeled by the scalar nonlinear Koiter plate equation. The fluid motion is modeled by the incompressible Euler equations defined on a changing domain. 
The interaction between the elastic plate and the fluid flow is 
captured through two mechanisms. The first is the dynamic loading effect on the structure described by the fluid pressure, and the second is the kinematic impermeability condition, which matches the normal velocities of the structure with the normal velocity of the flow.  
The Koiter model, which is used to describe the moving interface, is a nonlinear two-dimensional equation proposed by Koiter to describe the deformations of thin elastic shells, and it was derived under two main assumptions. The first one is that points at the normal to the mid-surface of the shell remain at the normal to the deformed mid-surface after deformation has occurred, and the second is that stress forces within the shell are tangential to the mid-surface. The non-linearity in the model reflects a significant coupling of flexural-longitudinal effects due to large deformations. We consider the case of a nonlinear Koiter-type elastic plate, in which the reference configuration is assumed to be flat. However, the analysis and the results can be extended to the general elastic shell case (non-flat reference configurations). For a rigorous justification of the model, see e.g.~\cite{CR}.

In this paper, our main result is the derivation of  a~priori estimates and the local in-time existence of strong solutions to the model. The a~priori estimates for the whole system are obtained by combining higher-order energy estimates for the nonlinear Koiter plate with pressure and vorticity estimates. It is important that the estimates are independent of the viscoelastic damping imposed on the plate. Local-in-time solutions for the regularized system (with plate viscoelastic damping $\nu >0$) are constructed via a fixed point argument, and solutions for the undamped system are then obtained using the a~priori estimates as the viscoelastic damping parameter $\nu$ goes to zero.

The critical step in obtaining higher-order energy estimates for the nonlinear Koiter plate involves calculating and estimating the elastic operator defined in~\eqref{elastico}. Additionally, we use the incompressible Euler equation, with the boundary conditions, to set up an elliptic equation for the pressure with Robin boundary data. Elliptic regularity then allows us to establish the appropriate estimates for the pressure. Finally, velocity estimates are obtained through div-curl type estimates.

The well-posedness of free boundary problems that couple lower-dimensional structures with viscous fluid has been an active area of research. Some of the first results along these lines were obtained by Desjardins~et~al \cite{CDEG, DEGL}, who showed the existence of weak solutions to a system coupling the  Navier-Stokes equation (NSE) to a strongly damped linear plate. Other existence results on weak solutions for models that couple the NSE to lower-dimensional structures have been obtained in~\cite{G,MC1,MC2, HLN,TW}. The mentioned results on the existence of weak solutions are valid until the possible self-intersection of the domain. Results on strong solutions of two-dimensional fluid interacting with a one-dimensional solid for large initial data have been obtained in~\cite{GH,GHL}. Additionally, a three-dimensional fluid interacting with a three-dimensional elastic body with structural damping has been considered in~\cite{IKLT1,IKLT2}, establishing the global results for small initial data. All the results mentioned above, except \cite{TW}, are for linear structure equations. The theory for nonlinear structures is considerably less developed, with only a few recent findings extending the theory of weak solutions to cases where the structure is modeled as a nonlinear generalized standard material; see e.g.~\cite{BKS}. Systems that model the interaction of viscous flows with a linear and nonlinear Koiter plate have been studied in~\cite{CS,MC3,MC4, L,LR,MS,MRR}. On the other hand, inviscid flow-structure systems have been examined in~\cite{BLW,CLW1,CLW2,W}. These models consider linearized potential flow coupled with a nonlinear structure on a fixed domain. 
The Euler equations coupled with a fourth-order linear plate have been addressed only recently in~\cite{KT1,KT2} and \cite{KNT} for incompressible and compressible cases, respectively.

The paper is organized as follows. In Section~\ref{2}, we begin by describing the model and reformulating the problem using ALE variables. We then state our main result in Theorem~\ref{main}, which provides a~priori bounds for existence. The a~priori estimates are obtained independently of the viscoelastic damping parameter in the plate $\nu >0$. While the construction scheme for the local-in-time solutions uses viscoelastic damping in an essential way, the a~priori estimates allow for the construction of solutions for the undamped system ($\nu=0$).
We end Section~\ref{2} with a description of the elastic operator~$L_K$. In Section~\ref{sec03}, we individually estimate the plate, pressure, and fluid velocity to prove Theorem~\ref{main}.

\section{Mathematical model and main results}\label{2}
\subsection{Model}
We consider a fluid-structure interaction model on an open domain $\Omega(t)\subset \mathbb{R}^3$, which depends on time $t\in [0,T]$, with $T>0$ specified below. The fluid is modeled by the incompressible 3D~Euler equations
  \begin{align}
  \begin{split}
  	& u_{t}+u\cdot \nabla u+\nabla p=0
  	 \\&
	 \nabla\cdot u=0
  	\end{split}
  	 \label{EQ01}
  \end{align}
in $\Omega(t)\times [0,T]=\bigcup_{t\in [0,T]}\Omega(t)$. For simplicity, we consider
$\Omega(0)=\Omega=\mathbb{T}^2\times[0,1]$, where $\mathbb{T}^2=\R^2/\Z^2$ is a two-dimensional torus.
In other words, we consider the 1-periodic boundary conditions in the horizontal directions. 
We denote the upper and lower parts of the boundary of $\Omega(0)$ as 
\begin{equation}\label{boundary}
	\Gamma_0=\mathbb{T}^2\times\{0\}
    \andand
	\Gamma_1=\mathbb{T}^2\times\{1\} 
.
\end{equation}
On $\Gamma_0$, we impose the slip boundary condition
\begin{equation}
	u\cdot N=0.
   \label{EQ04}
\end{equation}
The nonlinear Koiter plate is located along the top boundary, and the fluid domain is given by
\begin{equation}
\Omega(t)=\{(x_1,x_2,x_3): x_3<1+w(x_1,x_2,t)\},
   \llabel{EQ62}
\end{equation}
where $w\colon\Gamma_1\times [0,T]\to\mathbb{R}$ is the vertical displacement of the plate, which satisfies the equation
\begin{equation}\label{plate}
	h w_{tt}+L_Kw  - \nu \Delta w_{t} = p|_{x_{3}=1+w} \inon{ on $\Gamma_1 \times [0,T]$}
	,
\end{equation}
where $h>0$ is the thickness of the nonlinear Koiter plate,
which is considered fixed,
$L_K$ is a nonlinear operator  and is the $L^{2}$ gradient of the Koiter energy described in Section~\ref{elastic} and $\nu\geq0$ is the damping parameter. Note that $p$
is evaluated at $(x_1,x_2,1+w(x_1,x_2,t))$, while $w$ satisfies the initial condition
  \begin{equation*}
	(w,w_t)|_{t=0}=(0,w_1)
	,
  \end{equation*}
where
\begin{equation}\label{EQ31}
\int_{\Gamma_1}w_1=0
  .
\end{equation}
Note that we are mainly interested in the case $\nu=0$; however, the parameter $\nu$
is needed in the construction of solutions.
This analysis applies with minimal changes to general initial data
$w(0) \in H^5$ allowing for more general initial
configuration of the interface. However, we assume $w(0)=0$ for simplicity.
  
We assume that the nonlinear plate evolves with the fluid and we thus have the velocity-matching condition on the interface $\Gamma_1(t)$
\begin{equation}\label{vmatch}
	w_t + u_1\partial_1w + u_2\partial_2w = u_3,
\end{equation}
where $\Gamma_1(t)=\{(x_1,x_2,1+w(x_1,x_2,t):(x_1,x_2)\in\mathbb{T}^2\}$.

\subsection{ALE change of variables}
Denote by $\psi\colon\Omega\to\mathbb{R}$ the harmonic extension of $1+w$ to the domain $\Omega=\Omega(0)$, i.e.,
\begin{alignat}{2}\label{psi}
	\Delta\psi&=0  \qquad &&\text { in } \Omega,
	\\
	\psi(x_1,x_2,1,t)&=1+w(x_1,x_2) \qquad &&\text { on } \Gamma_1
	\\
	\psi(x_1,x_2,0,t)&=0 \qquad &&\text { on } \Gamma_0
	.\label{psii}
\end{alignat} 
Next, we define $\eta\colon\Omega\times[0,T]\to\Omega(t)$ as
\begin{equation}\label{eta}
\eta(x_1,x_2,x_3,t)=(x_1,x_2,\psi(x_1,x_2,x_3,t))\comma(x_1,x_2,x_3)\in\Omega, 
\end{equation}
which represents the ALE change of variables. Note that
\begin{equation}
\nabla\eta=\begin{bmatrix}
1 & 0 & 0\\
0 & 1 & 0\\
\partial_1\psi & \partial_2\psi & \partial_3\psi 
\end{bmatrix}.
\end{equation}
Denote $E=(\nabla\eta)^{-1}$, which is explicitly given by
\begin{equation}\label{edef}
E=\frac{1}{J}F=\begin{bmatrix}
1 & 0 & 0\\
0 & 1 & 0\\
-\partial_1\psi /\partial_3\psi & -\partial_2\psi /\partial_3\psi & 1/\partial_3\psi 
\end{bmatrix},
\end{equation}
where 
\begin{equation}
J=\partial_3\psi
\end{equation}
and 
\begin{equation}\label{fdef}
F=\begin{bmatrix}
\partial_3\psi & 0 & 0\\
0 & \partial_3\psi &0\\
-\partial_1\psi &  -\partial_2\psi & 1
\end{bmatrix}
\end{equation}
is the transpose of the cofactor matrix. Note that $F$ satisfies the Piola identity 
\begin{equation}\label{piola}
\partial_iF_{ij}=0\comma j=1,2,3
\end{equation}
using \cite[page 39]{piola} or a direct verification. Next, denote by
\begin{equation}\label{alevp}
v(x_1,x_2,x_3,t)=u(\eta(x_1,x_2,x_3,t),t)
\andand
q(x_1,x_2,x_3,t)=p(\eta(x_1,x_2,x_3,t),t)
\end{equation}
the ALE velocity and pressure. With this change of variables, the system \eqref{EQ01} becomes
\begin{equation}\label{eulerr}
\begin{split}
&\partial_tv_i+v_1E_{j1}\partial_jv_i+v_2E_{j2}\partial_j v_i+\frac{1}{\partial_3\psi}(v_3-\psi_t)\partial_3v_i+E_{ki}\partial_kq=0\\
&E_{ki}\partial_k v_i=0
\end{split}
\end{equation}
in $\Omega\times[0,T],$ where we used $E_{j3}\partial_jv_i=(1/\partial_3\psi)\partial_3v_i$, for $i=1,2,3$. The initial condition is given by
\begin{equation}
v\big|_{t=0}=v_0.
\end{equation}
On the other hand, the boundary condition on the bottom boundary reads
\begin{equation}\label{vg0}
v_3=0
\onon{\Gamma_0},
\end{equation}
while the condition \eqref{vmatch} may be rewritten as
 \begin{equation}
F_{3i}v_i = w_t
\inon{on $\Gamma_1$}
.
\label{EQ21}
\end{equation}
The plate equation takes the form
\begin{equation}\label{plattee}
h w_{tt}+L_Kw  - \nu \Delta w_{t}  =q \onon{\Gamma_1}
,
\end{equation}
where $q$ is as in~\eqref{alevp}.
As in \cite[Section~4]{KT1}, we have
  \begin{equation}\label{PressureNormalized}
   \int_{\Gamma_1}q=0
   \comma t\in[0,T]
   ,
  \end{equation}
and then from \eqref{plattee}, we get
\begin{equation}\label{EQ03}
\int_{\Gamma_1}w=0
    \comma t\in[0,T]
   ,
\end{equation}
assuming
$\int_{\Gamma_1} w_1=0$.

With the ALE coordinates introduced, all the estimates from this point
on are performed, unless specified otherwise,  on the fixed domain
$\Omega$, with the top and bottom boundaries $\Gamma_1$ and $\Gamma_0$, respectively. \\
 We now state the main result of our paper in the form of the following theorem
that summarizes our a~priori estimates for the local in-time existence for the system~\eqref{eulerr}--\eqref{plattee}.

\cole
\begin{Theorem}\label{main}
Let $ \nu \geq 0$. Assume that
$(v,q,w)$ is a $C^{\infty}([0,T_{0}] \times \Omega)$ solution to the initial boundary value problem \eqref{eulerr}--\eqref{PressureNormalized} on a time interval $[0,T_0]$, where $T_0>0$, such that
\begin{align}
\begin{split}
\Vert v_0\Vert_{H^{3.5}},
\Vert w_1\Vert_{H^{3}(\Gamma_1)}
\leq M
,
\end{split}
\llabel{EQ23}
\end{align}
where $M\geq1$.
Then 
$v$, $w$, and $q$ satisfy
\begin{align} 
\begin{split} 
&\Vert v\Vert_{H^{3.5}}, 
\Vert w\Vert_{H^{5}(\Gamma_1)},
\Vert w_t\Vert_{H^{3}(\Gamma_1)},
\Vert \psi\Vert_{H^{5.5}},
\Vert \psi_t\Vert_{H^{3.5}},
\Vert E\Vert_{H^{4.5}}
\leq C_0M
\comma t\in[0,\min\{T,T_0\}]
\end{split}
\llabel{EQ24}
\end{align}
and
\begin{align} 
\begin{split} 
&\Vert v_t\Vert_{H^{1.5}(\Gamma_1)}, 
\Vert w_{tt}\Vert_{H^{1}(\Gamma_1)}, 
\Vert q\Vert_{H^{2.5}}
\leq K
\comma t\in[0,\min\{T,T_0\}]
,
\end{split}
\llabel{EQ25}
\end{align}
where
$C_0>0$ is a constant that is independent of the parameter $\nu \geq0$, while $K$ and $T$ are constants that depend on~$M$, but not on~$\nu$.
\end{Theorem}
\colb

\subsection{The elastic operator}\label{elastic}
In this section, we derive an expression for the Koiter energy of the nonlinear Koiter plate by following~\cite{CR}, where this model is justified. For the rest of this paper, Greek indices and exponents take values in the set $\{1,2\}$ and Latin indices and exponents take values in the set $\{1,2,3\}$. 
The tangential vectors of the deformed configuration are given by
  \begin{equation*}
    a_\alpha(w)=\partial_\alpha\eta|_{x_3=1}= e_\alpha+n\partial_\alpha w
    \comma \alpha=1,2
    ,
  \end{equation*}
where $e_1=(1,0,0)$ and $e_2=(0,1,0)$ represent the unit tangential
vectors of the undeformed boundary~$\Gamma_1$
and $a_3=n=(0,0,1)$ is the unit normal vector to~$\Gamma_1$.
The components of the first fundamental form of the deformed configuration read
  \begin{equation*}
   a_{\alpha\beta}(w)
     =a_\alpha(w)\cdot a_\beta(w)
     =\delta_{\alpha\beta}+\partial_\alpha w\partial_\beta w
    ,
  \end{equation*}
where it is understood that $\alpha,\beta=1,2$.
We define the change of the metric tensor $G(w)$ as
\begin{equation*}
G_{\alpha\beta}(w)=a_{\alpha\beta}(w)-e_{\alpha}\cdot e_{\beta}=\partial_\alpha w\partial_\beta w.
\end{equation*}
The non-normalized normal vector to the deformed configuration is given by
  \begin{equation*}
  n(w)=a_1(w)\times a_2(w)=n-e_1\partial_1 w-e_2\partial_2 w
   ,
  \end{equation*}
which then leads to the definition \begin{equation}
\displaystyle a_3(w)=\frac{n(w)}{|n(w)|}.
   \llabel{EQ63}
\end{equation}
As in \cite{MS}, we define the tensor 
  \begin{equation}\label{curve}
   R_{\alpha\beta}=\frac{1}{|a_1\times a_2|}\partial_\alpha a_\beta(w)\cdot n(w)-\partial_\alpha a_\beta\cdot n=\partial^2_{\alpha\beta}w,
  \end{equation}
representing a non-normalized variant of the second fundamental form measuring a change in curvature;
here, $\partial_\alpha a_\beta$ represents the covariant components of the curvature tensor of the reference configuration.
Also,  in \eqref{curve}, $\partial_{\alpha}a_{\beta}=\partial_{\alpha}a_{\beta}(0)$
and $|a_1\times a_2|=|a_1(0)\times a_2(0)|=|e_1\times e_2|$ are understood.
For a treatment of the normalized version of the tensor~$R$, see Section~\ref{sec04}.

Finally, we define the elasticity tensor
  \begin{equation*}
   \mathcal{A}E=\frac{4\lambda\mu}{\lambda+2\mu}(A:E)A+4\mu AEA
   \comma E\in\Sym\mathbb{R}^{2\times 2}
   ,
  \end{equation*}
where
$A:E=\Tr(A E^{T})$.
Here, $A$ is the contravariant metric tensor associated to $\Gamma_1$
and $\lambda, \mu>0$  are the Lam\'{e} constants; the components of
the matrix $A$ are given by
  \begin{equation}\llabel{contra}
   a^{\alpha\beta}=a^{\alpha}\cdot a^{\beta},
  \end{equation}
where the three vectors $a^i$ are defined by the relation
$a_j\cdot a^i=\delta_{ij}$, for $i,j=1,2,3$,
and $a^3=a_3=n$.
For our geometry, we obtain $\displaystyle a^i=a_i$ for $i=1,2,3$.
We then define the Koiter energy $\mathcal{E}_{K}$ as 
  \begin{equation}\label{energy}
  \begin{split}
   \mathcal{E}_{K} &=\frac{h}{4}\int_{\Gamma_1}\mathcal{A}G(w(t,\cdot)): G(w(t,\cdot)) + \frac{h^3}{48}\int_{\Gamma_1}\mathcal{A}R(w(t,\cdot)): R(w(t,\cdot))  \\
   & =\frac{h}{4}\int_{\Gamma_1}a^{\alpha\beta\sigma\tau}G_{\alpha\beta}(w(t,\cdot)) G_{\sigma\tau}(w(t,\cdot))
   	+ \frac{h^3}{48}\int_{\Gamma_1}a^{\alpha\beta\sigma\tau}R_{\alpha\beta}(w(t,\cdot)) R_{\sigma\tau}(w(t,\cdot)).
  \end{split}
  \end{equation}
 In order to simplify the notation, we introduce bilinear forms connected to the membrane and bending effects in the variational formulation as
  \begin{equation}\label{elastics}
  \begin{split}
   \am(t,w,\xi)&=\frac{h}{2}\int_{\Gamma_1}\mathcal{A}G(w(t,\cdot)):\left(G'(w(t,\cdot))\xi \right)=\frac{h}{2}\int_{\Gamma_1}a^{\alpha\beta\sigma\tau}G_{\alpha\beta}(w(t,\cdot))\left(G'_{\sigma\tau}(w(t,\cdot))\xi\right)\\
   \ab(t,w,\xi)&=\frac{h^3}{24}\int_{\Gamma_1}\mathcal{A}R(w(t,\cdot)):\left(R'(w(t,\cdot))\xi\right) =\frac{h^3}{24}\int_{\Gamma_1}a^{\alpha\beta\sigma\tau}R_{\alpha\beta}(w(t,\cdot))\left(R'_{\sigma\tau}(w(t,\cdot))\xi\right),
  \end{split}
  \end{equation}
respectively,
for $\xi\in H^2(\Gamma)$; here, $G'$ and $R'$ denote the Fr\'{e}chet derivatives of $G$ and $R$ respectively, and are given by
  \begin{align}
   \begin{split}
   &
   (R'(w)\xi)_{\alpha\beta}=\partial_{\alpha\beta}\xi
   ,
   \\&
   (G'(w)\xi)_{\alpha\beta}=\partial_{\alpha}\xi\partial_\beta w+\partial_{\alpha}w\partial_\beta \xi
   ,
  \end{split}
   \label{frechet}
  \end{align}
and where $\alpha,\beta,\sigma,\tau$ range over~$1,2$.
Also, the tensor $a^{\alpha\beta\sigma\tau}$ is defined as
\begin{equation}\label{lame}
\begin{split}
a^{\alpha\beta\sigma\tau}&=a^{\alpha\beta}a^{\sigma\tau}\frac{4\lambda\mu}{\lambda+2\mu}+4\mu a^{\alpha\sigma}a^{\beta\tau}=\delta_{\alpha\beta}\delta_{\sigma\tau}\frac{4\lambda\mu}{\lambda+2\mu}+4\mu \delta_{\alpha\sigma}\delta_{\beta\tau}.
\end{split}
\end{equation}
The elastodynamics of the plate is given by the variational formulation
\begin{equation}
\begin{split}
h\rho_s\frac{d}{dt}\int_{\Gamma_1}\partial_tw(t,\dot)\xi dy+\am(t,w,\xi)+\ab(t,w,\xi)
  +\nu\int_{\Gamma_1}\nabla w_t\cdot\nabla\xi
=\int_{\Gamma_1}g\xi dy
    \comma \xi\in H^{2}(\Gamma_1)
\end{split}
\end{equation}
on the time interval $(0,T)$,
where $\rho_s$ is the structure density, $g$ is the density of area force acting on the structure. 
We denote the elasticity operator by $L_K$, formally given by
  \begin{equation}\label{elastico}
   \langle L_Kw,\xi\rangle=\am(t,w,\xi)+\ab(t,w,\xi)
   ,
  \end{equation}
where $\langle \cdot, \cdot\rangle$ denotes the duality pairing with respect to~$L^{2}$.

To obtain an explicit expression for $L_K$ for our particular configuration, we first rewrite
  \begin{equation}\label{needam}
  \begin{split}
  \am(t,w,\xi)&=\frac{ha^{\alpha\beta\sigma\tau}}{2}\left[\int_{\Gamma_1}\partial_\alpha w\partial_\beta w\partial_\sigma \xi\partial_\tau w+\int_{\Gamma_1}\partial_\alpha w\partial_\beta w\partial_\sigma w\partial_\tau \xi\right]\\
  &=-\frac{ha^{\alpha\beta\sigma\tau}}{2}\left[\int_{\Gamma_1}\xi\left(\partial_\sigma\left(\partial_\alpha w\partial_\beta w\partial_\tau w\right)\right) +\int_{\Gamma_1}\xi\left(\partial_\tau\left(\partial_\alpha w\partial_\beta w\partial_\sigma w \right)\right)\right]\\
  &=(L_{\text{m}}w,\xi),
  \end{split}
 \end{equation}
 where
 \begin{equation}
 L_{\text{m}}w=-\frac{ha^{\alpha\beta\sigma\tau}}{2}\left(\partial_\sigma\left(\partial_\alpha w\partial_\beta w\partial_\tau w\right)+\partial_\tau\left(\partial_\alpha w\partial_\beta w\partial_\sigma w \right)\right).
 \end{equation}
Similarly,
 \begin{equation}\label{needab}
 \begin{split}
  \ab(t,w,\xi)
  &=\frac{h^3a^{\alpha\beta\sigma\tau}}{24}\int_{\Gamma_1}\partial^2_{\alpha\beta}w\partial^2_{\sigma\tau}\xi
  =\frac{h^3a^{\alpha\beta\sigma\tau}}{24}\int_{\Gamma_1}\xi\partial^4_{\alpha\beta\sigma\tau}w
  =(L_{\text{b}}w,\xi),
  \end{split}
 \end{equation}
 where
 \begin{equation}
 L_{\text{b}}w=\frac{h^3a^{\alpha\beta\sigma\tau}}{24}\partial^4_{\alpha\beta\sigma\tau}w.
 \end{equation}
Therefore, combining \eqref{elastico}, \eqref{needam}, and \eqref{needab}, we obtain
 \begin{align}
   \begin{split}
 L_Kw&=L_{\text{m}}w+L_{\text{b}}w
  \\&
 =-\frac{ha^{\alpha\beta\sigma\tau}}{2}\left(\partial_\sigma\left(\partial_\alpha w\partial_\beta w\partial_\tau w\right)+\partial_\tau\left(\partial_\alpha w\partial_\beta w\partial_\sigma w \right)\right)+\frac{h^3a^{\alpha\beta\sigma\tau}}{24}\partial^4_{\alpha\beta\sigma\tau}w.
  \end{split}
  \label{lkexplicit}
 \end{align}
 \colb

\section{A~priori estimates}\label{sec03}
In this section, we provide a~priori estimates for $v$, $q$, and $w$
in appropriate Sobolev spaces and
prove Theorem~\ref{main}. We begin by deriving preliminary estimates on $E$, $F$, and~$J$.

\cole
\begin{Lemma}
\label{L01}
Let $\epsilon\in(0,1/2]$, and assume that
\begin{align}
\begin{split}
&
\Vert v\Vert_{H^{3.5}},
\Vert w\Vert_{H^{5}(\Gamma_1)},
\Vert w_t\Vert_{H^{3}(\Gamma_1)}
\leq C_0M
\comma t\in[0,T_0]
,
\end{split}
\llabel{EQ33}
\end{align}
where $M\geq1$ is as in the statement of Theorem~\ref{main}. Then
we have
\begin{equation}\label{matest}
\Vert E-I\Vert_{H^{2.5}},
\Vert F-I\Vert_{H^{2.5}},
\Vert J-1\Vert_{H^{2.5}},
\Vert J-1\Vert_{L^{\infty}} 
\leq \epsilon
\comma t\in [0,T_0]
,
\end{equation}
where 
\begin{equation}
0\leq T_0\leq\frac{\epsilon}{C M^{3}}
,
\llabel{EQ36}
\end{equation}
and $C$ is a sufficiently large constant depending on~$C_0$.
\end{Lemma}
\colb

Before the proof, note that by the definitions of 
$\psi$ and
$\eta$ in \eqref{psi}--\eqref{psii}
and \eqref{eta}
we have
  \begin{align}
   \begin{split}
    \Vert \eta\Vert_{H^{5.5}}
    \lec
    \Vert \psi\Vert_{H^{5.5}}
    \lec
    \Vert w\Vert_{H^{5}(\Gamma_1)}
   \end{split}
   \label{EQ42}
  \end{align}
and
  \begin{align}
   \begin{split}
    \Vert \eta_t\Vert_{H^{3.5}}
    \lec
    \Vert \psi_t\Vert_{H^{3.5}}
    \lec
    \Vert w_t\Vert_{H^{3}(\Gamma_1)}
    ,
   \end{split}
   \label{EQ43}
  \end{align}
and both far right sides are bounded by constant multiples of~$M$.

Also, we have
  \begin{align}
   \begin{split}
   \Vert J\Vert_{H^{4.5}}
   =
   \Vert \partial_{3}\psi\Vert_{H^{4.5}}
   \lec
   \Vert \psi\Vert_{H^{5.5}}
   \lec
   \Vert w\Vert_{H^{5}(\Gamma_1)}
   \end{split}
   \label{EQ44}
  \end{align}
and
  \begin{equation}
   \Vert J_t\Vert_{H^{2.5}}
   \lec
   \Vert \psi_t\Vert_{H^{3.5}}
   \lec
   \Vert \eta_t\Vert_{H^{3.5}}
   \lec
   \Vert w_t\Vert_{H^{3}(\Gamma_1)}
   ,
   \label{EQ45}
  \end{equation}
with both right sides bounded by a constant multiple of~$M$.

\begin{proof}[Proof of Lemma~\ref{L01}]
Since $E=(\nabla \eta)^{-1}$, we have
  \begin{equation}
   E_t = - E\nabla \eta_t E
   ,
   \label{EQ32}
  \end{equation}
where the right-hand side is understood as a product of three matrices. Next, we observe
  \begin{align}
  \begin{split}
   \partial_{t}(E-I)
     &=
     (E-I)\nabla \eta_t (E-I)
     +      (E-I)\nabla \eta_t
     +      \nabla \eta_t (E-I)
     +      \nabla \eta_t
     ,
  \end{split}
   \label{EQ101}
  \end{align}
from where we obtain that
the function $y=\Vert E-I\Vert_{H^{2.5}}$
satisfies
  \begin{align}
  \begin{split}
   y &
   \lec
   \int_{0}^{t} (y^2+1) M \,ds
   ,
  \end{split}
   \label{EQ125}
  \end{align}
where we used~\eqref{EQ43}. Using a barrier argument,
we get $y\lec \epsilon$ provided
$T\leq \epsilon/ C M$,
and the bound on the first term in
\eqref{matest} is established.
Similarly,
we have
  \begin{align}
   \begin{split}
   \Vert J-1\Vert_{H^{2.5}}
   \lec 
   \left\Vert \int_{0}^{t} J_t\,ds\right\Vert_{H^{2.5}}
   \lec
   M T
   .
   \end{split}
   \llabel{EQ47}
  \end{align}
The bound 
for     $\Vert F-I\Vert_{H^{2.5}}$
follows immediately from those on
$\Vert E-I\Vert_{H^{2.5}}$
and $    \Vert J-1\Vert_{H^{2.5}}$
by using $F = J E $.
\end{proof}

Note that the value of $\epsilon$ in Lemma~\ref{L01} is fixed in the pressure
estimates and then further restricted in the conclusion of a~priori bounds.

Also, by the definitions of $E$ and $F$ in \eqref{edef} and \eqref{fdef},
respectively,
we have
  \begin{equation}
   \Vert E\Vert_{H^{4.5}},
   \Vert F\Vert_{H^{4.5}}
   \leq
   P(\Vert w\Vert_{H^5(\Gamma_1)})
   \label{EQ48}
  \end{equation}
and
  \begin{equation}
   \Vert E_t\Vert_{H^{2.5}},
   \Vert F_t\Vert_{H^{2.5}}
   \leq
    P(\Vert w\Vert_{H^{5}(\Gamma_1)},
      \Vert w_t\Vert_{H^{3}(\Gamma_1)}
    )
   .
   \label{EQ49}
  \end{equation}
Above and in the sequel, the symbol $P$ denotes a generic polynomial of its arguments.
It is assumed to be nonnegative and is allowed to change from inequality to inequality.

\colb

\subsection{The plate estimates}\label{platesec}
We now proceed to derive estimates for the plate. For $f\in C^{m}(\mathbb{R}^2)$, $\gamma=(\gamma_1,\gamma_2)\in \mathbb{N}_0^2$, and $|\gamma|=\sum_{i=1}^2\gamma_i$, denote by $\partial^{\gamma}f$ the partial derivative of order $\gamma$ of $f$, which is defined as
$
\partial^{\gamma}f=\partial_{x_1}^{\gamma_1}\partial_{x_2}^{\gamma_2}f
$.
With $\gamma\in \mathbb{N}_0^3$, where $|\gamma|=3$, we test the plate equation \eqref{plate} with $-\partial^{2\gamma} w_t$, obtaining
\begin{equation}\label{plateen}
h(\partial^\gamma w_{tt},\partial^\gamma w_t)-(L_Kw,\partial^{2\gamma}w_t)
 +\nu(\partial^{\gamma}\nabla w_t\cdot \partial^{\gamma}\nabla w)
=-(\partial^{2\gamma}q,w_t)
.
\end{equation}
We now take a closer look at the term $(L_Kw,\partial^{2\gamma}w_t)$. Using \eqref{elastics} and \eqref{elastico}, we write
 \begin{equation}\label{d6}
 	\begin{split}
   		-(L_Kw,\partial^{2\gamma}w_t)&=-\frac{h}{2}\int_{\Gamma_1} \mathcal{A}G(w):G'(w)\partial^{2\gamma} w_t-\frac{h^3}{24}\int_{\Gamma_1}\mathcal {A}R(w):R'(w)\partial^{2\gamma}w_t\\
	  	&=-\frac{h}{2}\int_{\Gamma_1}a^{\alpha\beta\sigma\tau}G_{\alpha\beta}(w)(G'(w)\partial^{2\gamma}w_t)_{\sigma\tau}
  		\\&\indeq
  		-\frac{h^3}{24}\int_{\Gamma_1}a^{\alpha\beta\sigma\tau}R_{\alpha\beta}(w)(R'(w)\partial^{2\gamma}w_t))_{\sigma\tau}.
 \end{split}
 \end{equation}
We handle each of the above terms below. Using \eqref{frechet} and integrating by parts, we obtain
  \begin{equation}\label{g}
	\begin{split}
		&-\frac{h}{2}\int_{\Gamma_1}a^{\alpha\beta\sigma\tau}G_{\alpha\beta}(w)(G'(w)\partial^{2\gamma}w_t)_{\sigma\tau}
		\\&\indeq
		= -\frac{a^{\alpha\beta\sigma\tau}h}{2}\int_{\Gamma_1}\partial_\alpha w\partial_\beta w(\partial_\sigma \partial^{2\gamma}w_t)\partial_\tau w-\frac{a^{\alpha\beta\sigma\tau}h}{2}\int_{\Gamma_1}\partial_\alpha w\partial_\beta w\partial_\sigma w\partial_\tau \partial^{2\gamma}w_t
		\\&\indeq
		\lec a^{\alpha\beta\sigma\tau}h\left(\int_{\Gamma_1}|D^4\partial_\alpha w| |\partial_\beta w\partial_\sigma w| |D^2\partial_\tau w_t|+\int_{\Gamma_1}|D^3\partial_\alpha w| |D\partial_\beta w| |\partial_\sigma w ||D^2\partial_\tau w_t|\right)\\&\indeq\indeq
		+a^{\alpha\beta\sigma\tau}h\int_{\Gamma_1}|D^2\partial_\alpha w| |D^2\partial_\beta w| |\partial_\sigma w| |D^2\partial_\tau w_t|
		\\&\indeq\indeq
		+a^{\alpha\beta\sigma\tau}h\int_{\Gamma_1}|D^2\partial_\alpha w| |D\partial_\beta w| |D\partial_\sigma w| |D^2\partial_\tau w_t|
		\\&\indeq
		\leq P(\|w\|_{H^5(\Gamma_1)},\|w_t\|_{H^3(\Gamma_1)})
		,
	\end{split}
\end{equation}
where $P$ is a polynomial that may vary line to line. Denoting
$\partial_{\alpha}=\partial_{x_{\alpha}}$,
we use \eqref{frechet} and \eqref{lame} to get
  \begin{equation}\label{r}
  \begin{split}
  &-\frac{h^3a^{\alpha\beta\sigma\tau}}{24}\int_{\Gamma_1}R_{\alpha\beta}(w)(R'(w)\partial^{2\gamma}w_t)_{\sigma\tau}=-\frac{h^3a^{\alpha\beta\sigma\tau}}{24}\int_{\Gamma_1}\partial_{\alpha\beta}w \partial^{2\gamma}\partial_{\sigma\tau}w_t\\&\indeq
  =\frac{h^3a^{\alpha\beta\sigma\tau}}{24}\int_{\Gamma_1}\left(\partial^\gamma\partial_{\alpha\beta}w\right)\left(\partial^\gamma\partial_{\sigma\tau}w_t\right)\\&\indeq
  =\frac{4\lambda\mu h^3}{24(\lambda+2\mu)}\delta_{\alpha\beta}\delta_{\sigma\tau}\int_{\Gamma_1}\left(\partial^\gamma\partial_{\alpha\beta}w\right)\left(\partial^\gamma\partial_{\sigma\tau}w_t\right)
  \\&\indeq\indeq
+\frac{4\mu h^3}{24}\delta_{\alpha\sigma}\delta_{\beta\tau}\int_{\Gamma_1}\left(\partial^\gamma\partial_{\alpha\beta}w\right)\left(\partial^\gamma\partial_{\sigma\tau}w_t\right)\\&\indeq
  =\frac{4\lambda\mu h^3}{24(\lambda+2\mu)}\int_{\Gamma_1}\left(\partial^\gamma\partial_{\alpha\alpha}w\right)\left(\partial^\gamma\partial_{\sigma\sigma}w_t\right)+\frac{4\mu h^3}{24}\int_{\Gamma_1}\left(\partial^\gamma\partial_{\alpha\beta}w\right)\left(\partial^\gamma\partial_{\alpha\beta}w_t\right)
  \\&\indeq
  =\frac{1}{2}\frac{d}{dt}\biggl(\frac{4\lambda\mu h^3}{24(\lambda+2\mu)}\|\partial^\gamma\partial_{\alpha\alpha}w\|^2_{L^2(\Gamma_1)} + \frac{4\mu h^3}{24}\sum_{\alpha,\beta=1}^2\|\partial^\gamma\partial_{\alpha\beta}w\|^2_{L^2(\Gamma_1)}\biggr).
  \end{split}
  \end{equation}
Now, inserting the estimates from  \eqref{d6}, \eqref{g}, and \eqref{r} into \eqref{plateen}, we have
 \begin{equation}\label{platestp}
 \begin{split}
  &
  \frac{1}{2}\frac{d}{dt}\Bigg(h\|D^3w_t\|^2_{L^2(\Gamma_1)}
     +\frac{4\lambda\mu h^3}{24(\lambda+2\mu)}\|D^3\partial_{\alpha\alpha}w\|^2_{L^2(\Gamma_1)}
     + \frac{4\mu h^3}{24}\sum_{\alpha,\beta=1}^2\|D^3\partial_{\alpha\beta}w\|^2_{L^2(\Gamma_1)}\Bigg)
   \\&\indeq\indeq
     + \nu \Vert D^4 w_t \Vert^{2}_{L^2(\Gamma_1)}
     \\&\indeq
     \leq-(D^6q,w_t)+P(\|w\|_{H^5(\Gamma_1)},\|w_t\|_{H^3(\Gamma_1)})
     .
\end{split}
\end{equation}

Next, we eliminate the higher order pressure term. To do so, we need to proceed as in \cite[Lemma~3.2]{KT1}. Rather than repeating the entire proof, we summarize the strategy below.
\begin{enumerate}
\item Our aim is to prove the inequality
 \begin{align}
   \begin{split}
   &
   \frac12
   \int  J \Lambda^{2.5} v_i \Lambda^{3.5} v_i \Big|_{t}
   -
   \frac12
   \int  J \Lambda^{2.5} v_i \Lambda^{3.5} v_i \Big|_{0}
   \\&\indeq
   \leq
   - 
   \int_{0}^{t}
   \int_{\Gamma_1} q\Lambda^{6}w_{t}\,d\sigma \,ds
   + 
   \int_{0}^{t}
        P(
          \Vert v\Vert_{H^{3.5}},
          \Vert q\Vert_{H^{2.5}},
          \Vert w\Vert_{H^{5}(\Gamma_1)},
          \Vert w_{t}\Vert_{H^{3}(\Gamma_1)}
         )\,ds
   ,
   \end{split}
   \label{EQ326}
  \end{align}
where 
\begin{equation}\label{l2}
	\Lambda=(I-\Delta_2)^{1/2},
   \llabel{EQ09}
\end{equation}
and $\Delta_2$ denotes the Laplacian in $x_1$ and $x_2$ variables with the periodic boundary conditions. After \eqref{EQ326} is established, it is simply added to~\eqref{platestp}.
This leads to a cancellation of the higher order pressure term.
\item After the cancellation, using \eqref{matest}, the first term on the left hand side of \eqref{EQ326} is estimated as
 \begin{align}
   \begin{split}
    - \int  J \Lambda^{2.5} v_i \Lambda^{3.5} v_i
    &
    \lec
    \Vert J\Vert_{L^\infty}
    \Vert \Lambda^{2.5}v\Vert_{L^2}
    \Vert \Lambda^{3.5}v\Vert_{L^2}
    \lec
    \Vert \Lambda^{2.5}v\Vert_{L^2}
    \Vert \Lambda^{3.5}v\Vert_{L^2}
   \\&
   \lec
    \Vert v\Vert_{L^2}^{1/3.5}
    \Vert v\Vert_{H^{3.5}}^{6/3.5}
   .
   \end{split}
   \label{EQ76}
  \end{align}
\item To prove \eqref{EQ326}, we first claim
  \begin{align}
   \begin{split}
   &
   \frac12
   \frac{d}{dt}
   \int  J \Lambda^{2.5} v_i \Lambda^{3.5} v_i
     =
      \frac12 \int J_t \Lambda^{2.5} v_i \Lambda^{3.5} v_i
     + \int 
          J \Lambda^{2.5} \partial_{t}v_i
            \Lambda^{3.5} v_i
     + \bar I
    ,
   \end{split}
   \label{EQ328}
  \end{align}
where
  \begin{equation}
  \begin{split}
   \bar I
   &= \frac12\int 
          J \Lambda^{2.5} v_i
            \Lambda^{3.5} \partial_{t}v_i
     -
     \frac12\int 
          J \Lambda^{2.5} \partial_{t}v_i
            \Lambda^{3.5} v_i\\
   &\leq
     P(
       \Vert v\Vert_{H^{3.5}},
       \Vert q\Vert_{H^{2.5}},
       \Vert w\Vert_{H^{5}(\Gamma_1)},
       \Vert w_{t}\Vert_{H^{3}(\Gamma_1)}
      )
   .
   \label{EQ327}
   \end{split}
  \end{equation}

\item Next, the equations \eqref{eulerr} are rewritten as
  \begin{align}
   \begin{split}
   &
    J\partial_{t} v_i
    + v_1 F_{j1} \partial_{j}v_i
    + v_2 F_{j2} \partial_{j}v_i
    + (v_3-\psi_t) \partial_{3} v_i
    + F_{ki}\partial_{k}q
    =0
    ,
    \\&
    F_{ki} \partial_{k}v_i=0
    .
   \end{split}
   \label{EQ54}
  \end{align}
Using \eqref{EQ54}$_1$
in the second term of \eqref{EQ328}, we obtain
  \begin{align}
   \begin{split}
   &
   \frac12
   \frac{d}{dt}
   \int  J \Lambda^{2.5} v_i \Lambda^{3.5} v_i
     \\&\indeq
     =
      \frac12 \int J_t \Lambda^{2.5} v_i \Lambda^{3.5}v_i
     + \int 
         \Bigl(
          J \Lambda^{2.5} (\partial_{t}v_i)
            - 
          \Lambda^{2.5}(J\partial_t v_i) 
         \Bigr) \Lambda^{3.5}v_i
    \\&\indeq\indeq
    - \sum_{m=1}^{2}\int \Lambda^{2.5}(v_m F_{jm}\partial_{j}v_i) \Lambda^{3.5} v_i
    - \int \Lambda^{2.5} 
            \bigl(
              (v_3-\psi_t)\partial_{3}v_i                                                          
            \bigr) \Lambda^{3.5} v_i
   \\&\indeq\indeq
    -\int \Lambda^{3}(F_{ki}\partial_{k}q)\Lambda^{3} v_i 
    + \bar I
    \\&\indeq
    = I_1 + I_2 + I_3 + I_4 + I_5 + \bar I
   .
   \end{split}
   \label{EQ56}
  \end{align}
\item The terms on the right hand side of \eqref{EQ56} are bounded as
  \begin{equation}
  \begin{split}
   I_1 + I_2 + I_3 + I_4 
   &\leq P(
          \Vert v\Vert_{H^{3.5}},
          \Vert q\Vert_{H^{2.5}},
          \Vert w\Vert_{H^{5}(\Gamma_1)},
          \Vert w_{t}\Vert_{H^{3}(\Gamma_1)}
         )
   \end{split}
   \label{EQ17}
   \end{equation}
and
  \begin{equation}
  \begin{split}
    I_5
    &\leq  - 
   \int_{\Gamma_1} q\Lambda^{6}w_{t}+P(
          \Vert v\Vert_{H^{3.5}},
          \Vert q\Vert_{H^{2.5}},
          \Vert w\Vert_{H^{5}(\Gamma_1)},
          \Vert w_{t}\Vert_{H^{3}(\Gamma_1)}
         )
   .
  \end{split}
   \label{EQ29}
   \end{equation}
Applying \eqref{EQ17} and \eqref{EQ29} to \eqref{EQ56} and integrating in time, we then obtain~\eqref{EQ326}.
\end{enumerate}

Adding \eqref{platestp} and \eqref{EQ326} the pressure terms cancel, and we get
\colb
\begin{equation}\label{platestp2}
\begin{split}
  &h\|D^3w_t\|^2_{L^2}+\frac{4\lambda\mu h^3}{24(\lambda+2\mu)}\|D^3\partial_{\alpha\alpha}w\|^2_{L^2}  + \frac{4\mu h^3}{24}\sum_{\alpha,\beta=1}^2\|D^3\partial_{\alpha\beta}w\|^2_{L^2} + \nu \Vert D^4 w_t \Vert^{2}_{L^2} 
  \\&\indeq
   \lec\|D^3w_t(0)\|+\|v\|_{L^2}^{2/7}\|v\|_{H^{3.5}}^{12/7}
   +\int_0^tP(\|w\|_{H^5(\Gamma_1)},\|w_t\|_{H^3(\Gamma_1)},\|q\|_{H^{2.5}})
   ,
\end{split}
\end{equation}
where the implicit constant is independent of~$\nu$.

\begin{Remark}
\label{R03}
{\rm
While in our paper we consider the nonlinear Koiter type plate, the analysis can also be extended to other type of structures. Namely,
we may treat:\\
\indent 1. Extensible plates:
\begin{align*}
w_{tt} + \Delta^{2} w - \nu \Delta w_{t} + \left(\int_{\Gamma_{1}}  |\nabla w |^{2}   \, dx \right) \Delta w = q 
 .
 \end{align*}
\indent 2. Koiter shells: This is the more general setting in which the reference configuration of the interface is not flat. In particular, we may consider a reference configuration parameterized by a $C^{3}$ injective mapping $\phi\colon \Gamma_{1} \to \mathbb{R}^{3} $ with $\Gamma_{1}= \mathbb{R}^{2} /\mathbb{Z}^{2}$. In this case, we may assume that the deformation happens primarily in the normal direction and is described by the mapping
 \begin{align}
 \phi_{w}(y,t) = \phi(y) + w(t,y) n(y)
    \comma y\in\Gamma_1
    .
 \end{align}
 The tangent vectors $a_{\alpha} $ in this case are no longer the standard vectors $e_{\alpha}$ and are space dependent. This gives rise to a space dependent tensor~$a^{\alpha \beta \gamma\delta}$.
The local existence theorem can also be obtained  in this case under additional assumptions on the initial configuration which guarantee nondegeneracy of the $H^{2}$-coercivity of the Koiter energy $\mathcal{E}_{K}$. Unlike in our situation, the time of existence also depends on the persistence of this nondegeneracy condition; see~\cite{MS}.
}
\end{Remark}

\subsection{Pressure estimates}
Here we establish the following pressure estimate.

\cole
\begin{Lemma}
\label{L03}
Under the conditions of Theorem~\ref{main}, we have
  \begin{align}
   \begin{split}
     \Vert q\Vert_{H^{2.5}}
     \leq
      P(
          \Vert v\Vert_{H^{3.5}},
          \Vert w\Vert_{H^{5}(\Gamma_1)},
          \Vert w_{t} \Vert_{H^{3}(\Gamma_1)}
       )
   .
   \end{split}
   \label{EQ78}
  \end{align}
\end{Lemma}
\colb

To obtain an equation for the pressure in the interior, we apply $F_{ji}\partial_j$ to \eqref{eulerr} and then use the Piola identity \eqref{piola}
to write
  \begin{equation}
\partial_j(F_{ji}E_{ki}\partial_kq)=\partial_j(\partial_tF_{ji}v_i)-\partial_j\left(\sum_{m=1}^{2}F_{ji}v_mE_{km}\partial_kv_i\right)-\partial\left(E_{ji}(v_3-\psi_t)\partial_3v_i\right)
\inon{in $\Omega$}
,
   \label{EQ05}
\end{equation}
using the divergence-free condition.

To obtain the boundary condition on
$\Gamma_0\cup\Gamma_1$, we multiply \eqref{eulerr} with~$F_{3i}$ obtaining
  \begin{align}
   \begin{split}
       F_{3i}E_{ki} \partial_{k}q
    &= - F_{3i}\partial_{t}v_i
       - F_{3i} v_1 E_{j1} \partial_{j}v_i
     \\&\indeq
       - F_{3i} v_2 E_{j2} \partial_{j}v_i
       - E_{3i} (v_3-\psi_t) \partial_{3} v_i
     \inon{on $\Gamma_0\cup\Gamma_1$}    
    ,
   \end{split}
   \label{EQ81}
  \end{align}
where we again employed $F_{ji} / \partial_{3} \psi=E_{ji}$ in the last term.
On $\Gamma_1$, we use \eqref{EQ21} and \eqref{plattee} and rewrite the first term on the right-hand side of \eqref{EQ81} as
  \begin{align}
  \begin{split}
   -F_{3i}\partial_{t} v_i 
    &= - \partial_{t}(F_{3i} v_i) + \partial_{t}F_{3i} v_i
    = - w_{tt} + \partial_{t}F_{3i} v_i
    \\&
    = h^{-1}L_Kw
      - h^{-1} \nu \Delta w_{t}
    - h^{-1}q + \partial_{t}F_{3i} v_i
  .
  \end{split}
   \llabel{EQ82}
  \end{align}
Thus, on $\Gamma_1$, the boundary condition \eqref{EQ81} becomes
a Robin boundary condition
  \begin{align}
   \begin{split}
   F_{3i}E_{ki}\partial_{k}q
   + h^{-1}q
   &=h^{-1} (L_Kw
    -  \nu \Delta w_{t} )
   + \partial_{t}F_{3i}v_i
       - F_{3i} v_1 E_{j1} \partial_{j}v_i
       - F_{3i} v_2 E_{j2} \partial_{j}v_i\\&\indeq
     	- E_{3i} (v_3-\psi_t) \partial_{3} v_i   
    = g_1
    \inon{on $\Gamma_1$}
   .
   \end{split}
   \label{EQ83}
  \end{align}
On $\Gamma_0$, we have $E=I$, and then the first term on the right hand side of \eqref{EQ81} vanishes. Thus we get
  \begin{align}
   \begin{split}
   F_{3i}E_{ki}\partial_{k}q
   =
       - F_{3i} v_1 E_{j1} \partial_{j}v_i
       - F_{3i} v_2 E_{j2} \partial_{j}v_i
       - E_{3i} (v_3-\psi_t) \partial_{3} v_i   
    = g_0
    \inon{on $\Gamma_0$}
   .
   \end{split}
   \label{EQ84}
  \end{align}
To estimate the pressure, we need the following statement from \cite{KT1} on the elliptic regularity for the Robin-Neumann problem.

\cole
\begin{Lemma}
\label{L08}
Assume that 
$d\in W^{1,\infty}(\Omega)$.
Let $1\leq l\leq 2$, and suppose that $\qqq$ is an $H^{l}$
solution of
  \begin{align}
   \begin{split}
   &\partial_{i}(d_{ij}\partial_{j}\qqq) = \div f
   \inon{in $\Omega$}
   ,
   \\&
   d_{mk}\partial_{k}\qqq N_{m} + u=g_1
   \inon{on $\Gamma_1$}
   ,
   \\&
   d_{mk}\partial_{k}\qqq N_{m}=g_0
   \inon{on $\Gamma_0$}
   .
   \end{split}
   \label{EQ85}
  \end{align}
If
  \begin{equation}
   \Vert d-I\Vert_{L^\infty}
   \leq \epsilon_0
   ,
   \label{EQ86}
  \end{equation}
where $\epsilon_0>0$  is sufficiently small, then
  \begin{equation}
   \Vert u\Vert_{H^{l}}
   \lec
   \Vert f\Vert_{H^{l-1}}
   +  \Vert g_1\Vert_{H^{l-3/2}(\Gamma_1)}
   +  \Vert g_0\Vert_{H^{l-3/2}(\Gamma_0)}
   .
   \label{EQ87}
  \end{equation}
\end{Lemma}
\colb

In this paper, we need an extension for $l\leq3$.

\cole
\begin{Lemma}
\label{L09}
Assume that 
$d\in W^{2,\infty}(\Omega)$.
Let $2\leq l\leq 3$, and suppose that $\qqq$ is an $H^{l}$
solution of
  \begin{align}
   \begin{split}
   &\partial_{i}(d_{ij}\partial_{j}\qqq) = \div f
   \inon{in $\Omega$}
   ,
   \\&
   d_{mk}\partial_{k}\qqq N_{m} + u=g_1
   \inon{on $\Gamma_1$}
   ,
   \\&
   d_{mk}\partial_{k}\qqq N_{m}=g_0
   \inon{on $\Gamma_0$}
   .
   \end{split}
   \label{EQ30}
  \end{align}
If
  \begin{equation}
   \Vert d-I\Vert_{L^\infty}
   \leq \epsilon_0
   ,
   \label{EQ34}
  \end{equation}
where $\epsilon_0>0$  is sufficiently small, then
  \begin{equation}
   \Vert u\Vert_{H^{l}}
   \leq
  P(\Vert f\Vert_{H^{l-1}},  \Vert g_1\Vert_{H^{l-3/2}(\Gamma_1)},  \Vert g_0\Vert_{H^{l-3/2}(\Gamma_0)}, \|d\|_{W^{2,\infty}})
   .
   \label{EQq38}
  \end{equation}
\end{Lemma}
\colb

\begin{proof}[Proof of Lemma~\ref{L09}]
We only show the a~priori estimate, i.e., we assume that $u$ is smooth
and satisfies~\eqref{EQ85}. The a~priori estimate can then be
justified
in a standard way using the difference quotients.
Let $n\in\{1,2\}$. Then, differentiating \eqref{EQ30},
we obtain that $U=\partial_{n}u$ satisfies
  \begin{align}
   \begin{split}
   &\partial_{i}(d_{ij}\partial_{j}U)
     = \div \partial_n f
        - \partial_{i}(\partial_{n}d_{ij}\partial_{j}u)
   \inon{in $\Omega$}
   ,
   \\&
   d_{mk}\partial_{k}U N_{m} + U
      = \partial_{n} g_1
         -    \partial_{n}d_{mk}\partial_{k}u N_{m} 
   \inon{on $\Gamma_1$}
   ,
   \\&
   d_{mk}\partial_{k}U N_{m}
     =
     \partial_{n}g_0
     - \partial_{n} d_{mk}\partial_{k}u N_{m}
   \inon{on $\Gamma_0$}
   .
   \end{split}
   \llabel{EQ39}
  \end{align}
Applying Lemma~\ref{L08}, and the fact that
$   \Vert Du  \Vert_{H^{l-5/2}(\Gamma_0\cup\Gamma_1)}
 \lec\Vert u  \Vert_{H^{l-1}}$, we get
  \begin{align}
    \begin{split}
   \Vert U\Vert_{H^{l-1}}
   &\lec
   \Vert \partial_n f\Vert_{H^{l-2}}
   +  \Vert\partial_n d_{ij}\partial_j u\Vert_{H^{l-2}}
   +  \Vert \partial_{n}g_1\Vert_{H^{l-5/2}(\Gamma_1)}
   \\&\indeq
   +  \Vert \partial_{n}g_0\Vert_{H^{l-5/2}(\Gamma_0)}
   +  \Vert \partial_{n}d_{mk}\partial_{k}u N_{m} \Vert_{H^{l-5/2}(\Gamma_0\cup\Gamma_1)}
   \\&
   \lec
   \Vert f\Vert_{H^{l-1}}
   +\Vert d\Vert_{W^{1,\infty}}\Vert u\Vert_{H^{l-1}}
   +  \Vert g_1\Vert_{H^{l-3/2}(\Gamma_1)}
   +  \Vert g_0\Vert_{H^{l-3/2}(\Gamma_0)}
   \\&\indeq
   +
   \Vert d\Vert_{W^{1,\infty}}
   \Vert Du  \Vert_{H^{l-5/2}(\Gamma_0\cup\Gamma_1)}
    \\&
   \lec
   (1+\Vert d\Vert_{W^{1,\infty}})
   \left( 
   \Vert f\Vert_{H^{l-1}}
   +  \Vert g_1\Vert_{H^{l-3/2}(\Gamma_1)}
   +  \Vert g_0\Vert_{H^{l-3/2}(\Gamma_0)}
   \right)
   .
  \end{split}
   \label{EQ38}
  \end{align}
Next, we rearrange \eqref{EQ30} to obtain
 \begin{equation}
 	d_{33}\partial_{33}^2u=\div f-\sum_{(i,j)\neq(3,3)}\partial_i(d_{ij}\partial_ju)-\partial_3d_{33}\partial_3u.
	\label{EQ39}
 \end{equation}
 Differentiating \eqref{EQ39} with $\partial_3$ we have
 \begin{equation} 
 	d_{33}\partial_{333}^3u=\div \partial_3 f-\sum_{(i,j)\neq(3,3)}\partial_3\partial_i(d_{ij}\partial_ju)-\partial_{33}^2d_{33}\partial_3u-2\partial_{3}d_{33}\partial_{33}^2u.
\end{equation}
Now, using \eqref{EQ38} with $l=3$, \eqref{EQ34} and Lemma~\ref{L08} we may write
\begin{equation}
\begin{split}
\|\partial_{333}u\|_{L^2}&\lec \Vert f\Vert_{H^{2}}+\|d\|_{L^{\infty}}\|U\|_{H^2}+\|d\|_{W^{2,\infty}}\|u\|_{H^1}+\|d\|_{W^{1,\infty}}\|u\|_{H^2}
 \\&
   \lec
   P(\Vert f\Vert_{H^{2}},  \Vert g_1\Vert_{H^{3/2}(\Gamma_1)},  \Vert g_0\Vert_{H^{3/2}(\Gamma_0)}, \|d\|_{W^{2,\infty}}).
   \end{split}
   \label{EQ40}
\end{equation}
Combining \eqref{EQ38} for $l=3$ and \eqref{EQ40}, we obtain the statement of the lemma for $l=3$. Using Lemma~\ref{L08} and interpolating between $l=2$ and $l=3$, we obtain the statement of the lemma.
\end{proof}

We now prove Lemma~\ref{L03}.

\begin{proof}[Proof of Lemma~\ref{L03}]
We apply the elliptic estimate \eqref{EQ87}
in $H^{2.5}$ to the equation \eqref{EQ05} with the boundary conditions
\eqref{EQ83}--\eqref{EQ84}, leading to
  \begin{align}
   \begin{split}
   \Vert q\Vert_{H^{2.5}}
   \lec
   P(\Vert f\Vert_{H^{1.5}},  \Vert g_1\Vert_{H^{1}(\Gamma_1)},  \Vert g_0\Vert_{H^{1}(\Gamma_0)}, \|d\|_{W^{2,\infty}})
   .
   \end{split}
   \label{EQ89}
  \end{align}
For the interior  term, we may write
  \begin{align}
   \begin{split}
   \Vert f\Vert_{H^{1.5}}
   &\lec
   \sum_{j=1}^{3}\Vert \partial_{t}F_{ji} v_i\Vert_{H^{\rr-2}}
   +   \sum_{j=1}^3\sum_{m=1}^{2}   \Vert F_{ji} v_m E_{km}\partial_{k}v_i\Vert_{H^{1.5}}
   +   \sum_{j=1}^{3}\Vert a_{ji} (v_3-\psi_t)\partial_{3}v_i\Vert_{H^{1.5}}
   \\&
   \leq P(
          \Vert v\Vert_{H^{3.5}},
          \Vert E\Vert_{H^{4.5}},
          \Vert F\Vert_{H^{4.5}},
          \Vert F_t\Vert_{H^{2.5}},
          \Vert \psi_t\Vert_{H^{3.5}},
         )
   .
   \end{split}
   \label{EQ90}
  \end{align}
Next, we observe that from \eqref{lkexplicit} we have
  \begin{equation}
  \|L_Kw\|_{H^1(\Gamma_1)}\leq P(\Vert w\Vert_{H^{5}(\Gamma_1)})
  .
  \end{equation}
For the boundary terms, we estimate
  \begin{align}
   \begin{split}
    &\Vert g_0\Vert_{H^{1}(\Gamma_0)}
     + \Vert g_1\Vert_{H^{1}(\Gamma_1)}
     \\&\indeq
     \lec
     \|L_Kw\|_{H^1(\Gamma_1)}
     + \| \Delta w_{t} \|_{H^1(\Gamma_1)}
     + \Vert F_t\Vert_{H^{1.5}(\partial\Omega)}
       \Vert v\Vert_{H^{1.5}(\partial\Omega)}
     \\&\indeq\indeq
     + \Vert F\Vert_{H^{2}(\partial\Omega)}
       \Vert v\Vert_{H^{2}(\partial\Omega)}
       \Vert E\Vert_{H^{2}(\partial\Omega)}
       \Vert \nabla v\Vert_{H^{1}(\partial\Omega)}
     \\&\indeq\indeq
     + \Vert E\Vert_{H^{2}(\partial\Omega)}
       (
         \Vert v\Vert_{H^{2}(\partial\Omega)}
         + \Vert \psi_t\Vert_{H^{2}(\partial\Omega)}
       )
      \Vert \nabla v\Vert_{H^{1}(\partial\Omega)}
    \\&\indeq
     \leq P(
          \Vert w\Vert_{H^{5}(\Gamma_1)},
          \Vert     w_{t} \Vert_{H^{3}(\Gamma_1)},
          \Vert F\Vert_{H^{4.5}},
          \Vert  F_t\Vert_{H^{2.5}},
          \Vert \psi_t\Vert_{H^{3.5}},
          \Vert v\Vert_{H^{3.5}}
         )
    ,
   \end{split}
   \label{EQ91}
  \end{align}
where $\partial\Omega=\Gamma_0\cup \Gamma_1$.
By combining 
\eqref{EQ89}--\eqref{EQ91}
and using \eqref{EQ48}--\eqref{EQ49}, we obtain~\eqref{EQ78}.
\end{proof}

\colb

\subsection{The velocity estimates and the conclusion}
In this section, we use the ALE~vorticity $
\zeta(x,t)=\curl u(\eta(x,t),t)$, where
$\omega=\curl u$, to get control of the velocity~$v$. The vorticity in
the ALE variables is given by \begin{equation} \zeta_i =
\epsilon_{ijk} E_{mj} \partial_{m} v_k \llabel{EQ95} \end{equation}
and satisfies \begin{align} \begin{split} \partial_{t}\zeta_i + v_1
E_{j1}\partial_{j} \zeta_i + v_2 E_{j2}\partial_{j} \zeta_i +
(v_3-\psi_t) E_{j3} \partial_{j} \zeta_i = \zeta_k
E_{mk}\partial_{m}v_i \comma i=1,2,3 .  \end{split} \llabel{EQ94}
\end{align} To obtain non-tangential estimates, we use the Sobolev
extension operator $f\mapsto \tilde f$, which is a
continuous operator $H^{k}(\Omega)\to
H^{k}(\mathbb{T}^2\times\mathbb{R})$ for all $k\in[0,6]$, where $k$ is
not necessarily an integer.  We require the extension to be such that
$\supp \tilde f$ vanishes in a neighborhood of
$\mathbb{T}^2\times(-1/2,3/2)^{\text c}$.  We now consider the
equation for $\theta=(\theta_1,\theta_2,\theta_3)$, given by
\begin{align} \begin{split} \partial_{t}\theta_i + \tilde v_1 \tilde
E_{j1}\partial_{j} \theta_i + \tilde v_2 \tilde E_{j2}\partial_{j}
\theta_i + (\tilde v_3-\tilde \psi_t) \tilde E_{j3} \partial_{j}
\theta_i = \theta_k \tilde E_{mk}\partial_{m}\tilde v_i \comma i=1,2,3
\end{split} \label{EQ98} \end{align}
 in
$\Omega_0=\mathbb{T}^2\times\mathbb{R}$, with the initial condition $
\theta(0)=\tilde \zeta(0) $.  Since $\theta$ obeys a transport type
equation, from the properties of the extension operator we have
\begin{equation} \theta(x,t) = 0 \comma (x,t)\in (\mathbb{T}^2\times
(-1/2,3/2)^{\text c} )\times [0,T] .  \llabel{EQ221} \end{equation}
Next, we define the quantity \begin{equation} X = \int_{\Omega_0}
|\Lambda_3^{2.5}\theta|^2 , \llabel{EQ122} \end{equation} where
\begin{equation} \Lambda_3 =(I-\Delta)^{1/2} \llabel{EQ108}
\end{equation} on ${\mathbb T}^2\times{\mathbb R}$, with the periodic boundary conditions in the $x_1$ and $x_2$ directions.  Repeating the
arguments of \cite[Lemma~3.7]{KT1}, we obtain \begin{align}
\begin{split} \frac{d}{dt}X \lec P(\Vert v\Vert_{H^{3.5}}, \Vert
w\Vert_{H^{5}(\Gamma_1)}, \Vert w_{t}\Vert_{H^{3}(\Gamma_1)} ) X
\comma t\in[0,T] .  \end{split} \label{EQ107} \end{align} Note that,
by the continuity properties of the Sobolev extensions, we have
\begin{equation} \Vert \zeta(0)\Vert_{H^{2.5}}^2 \lec X(0) \lec \Vert
\zeta(0)\Vert_{H^{2.5}}^2 .  \llabel{EQ124} \end{equation}
From \cite[Lemma~3.7]{KT1}, we 
also
obtain the inequality that relates $X$
to the vorticity, given by \begin{equation} \Vert
\zeta\Vert_{H^{2.5}}^2 \lec X .  \label{EQ35} \end{equation} To prove \eqref{EQ35}, we consider the
difference $\sigma=\zeta-\theta$, which satisfies \begin{align}
\partial_{t}\sigma_i + v_1 E_{j1}\partial_{j} \sigma_i + v_2
E_{j2}\partial_{j} \sigma_i + (v_3-\psi_t) E_{j3} \partial_{j}
\sigma_i = \sigma_k E_{mk}\partial_{m} v_i \inon{in $\Omega$} ,
\llabel{EQ102} \end{align} for~$i=1,2,3$.  This leads to the energy
equality \begin{align} \begin{split} \frac12 \frac{d}{dt} \int
|\sigma|^2 &= - \sum_{m=1}^{2} \int v_m E_{jm}\sigma_i \partial_{j}
\sigma_i - \int (v_3-\psi_t) E_{j3} \sigma_i \partial_{j} \sigma_i +
\int \sigma_k E_{mk}\partial_{m}v_i \sigma_i \\& = \frac12
\sum_{m=1}^{2} \int \Bigl(\partial_{j}( v_m E_{jm}) + \partial_{j}
(v_3-\psi_t)\Bigr)|\sigma|^2 + \int \sigma_k E_{mk}\partial_{m}v_i
\sigma_i , \end{split} \label{EQ103} \end{align} where the boundary
terms vanish since \begin{equation} \sum_{m=1}^{2} v_m E_{3m} + (
v_3-\psi_t) E_{33} = 0 \inon{on $\partial\Omega$}.  \llabel{EQ37}
\end{equation} From \eqref{EQ103} and $\sigma(0)=0$ in $\Omega$, we
obtain $\zeta=\theta$ in $\Omega$, and thus \eqref{EQ35} holds.\\
Lastly, we proceed to prove Theorem~\ref{main} and close our a~priori
estimates. 

 \begin{proof}[Proof of Theorem~\ref{main}] 
To establish the a~priori estimates and recover the full regularity of the velocity~$v$,
we employ the div-curl estimate from \cite{BB}, which provide control of the full Sobolev norm of the velocity given information on the divergence, curl, and the trace of the normal component of the velocity.
In particular,
  \begin{equation}\label{divcurl}
   \Vert v\Vert_{H^{3.5}}
   \lec
   \Vert \curl v\Vert_{H^{2.5}}
   +
   \Vert \div v\Vert_{H^{2.5}}
   +
   \Vert v\cdot N\Vert_{H^{3}(\Gamma_0\cup\Gamma_1)}
   +
   \Vert v\Vert_{L^2}
   .
   \end{equation}
We shall use \eqref{EQ35} above to control the curl and the plate-Euler tangential estimates  in \eqref{platestp2} from Section~\ref{sec03} in order control the normal component of the velocity at the boundary via the kinematic boundary condition~\eqref{vmatch}. Finally, for the divergence, we utilize the variable divergence-free condition.
   
For the curl of the velocity, we write
  \begin{align} \begin{split} \Vert \curl v\Vert_{H^{2.5}} &\lec
  \sum_{i} \Vert \epsilon_{ijk} E_{mj} \partial_{m} v_k \Vert_{H^{2.5}}
  + \sum_{i} \Vert \epsilon_{ijk} \partial_{m} v_k (E_{mj}-\delta_{mj})
\Vert_{H^{2.5}} \lec \Vert \zeta \Vert_{H^{2.5}} + \epsilon \Vert
v\Vert_{H^{3.5}} ,  \end{split} \llabel{EQ40} \end{align}
from where we estimate by
\eqref{EQ107} and \eqref{EQ35}, 
\begin{align}
\begin{split} \Vert \curl v\Vert_{H^{2.5}} \lec \Vert
\zeta(0)\Vert_{H^{2.5}} + \epsilon \Vert v\Vert_{H^{3.5}} +
\int_{0}^{t} P( \Vert v\Vert_{H^{3.5}}, \Vert
w\Vert_{H^{5}(\Gamma_1)}, \Vert w_{t}\Vert_{H^{3}(\Gamma_1)} ) X^{1/2}
\,ds .  \end{split} \label{EQ129} \end{align}
To treat the divergence,
we use the divergence-free condition to obtain
\begin{align}
\begin{split} \Vert \div v\Vert_{H^{2.5}} &= \Vert
(E_{ki}-\delta_{ki})\partial_{k} v_i \Vert_{H^{2.5}} \lec \epsilon
\Vert v\Vert_{H^{3.5}} , \end{split} \label{EQ131} \end{align}
while
for the bottom boundary term we have $v_3|_{\Gamma_0}
=0$. Employing \eqref{EQ21} on $\Gamma_1$, we get
\begin{align} \begin{split} \|v_3\|_{H^{3}}\lesssim
\epsilon\|w\|_{H^5(\Gamma_1)}\|v\|_{H^{3.5}}+\epsilon\|v\|_{H^{3.5}}+\|w_t\|_{H^{3}(\Gamma_1)}
, \end{split} \label{EQ132} \end{align} where $\epsilon\in(0,1]$ is
arbitrary.  Next, we bound the second term on the right side of
\eqref{platestp2} as
  \begin{equation}\label{wd}
  \begin{split}
\|v\|_{L^2}^{2/7}\|v\|_{H^{3.5}}^{12/7}\leq&~
\epsilon_1\|v\|^2_{H^{3.5}}+C_{\epsilon_1}\|v\|_{L^2}^2\leq
\epsilon_1\|v\|^2_{H^{3.5}}
+C_{\epsilon_1}\|v_0\|^2_{L^{2}}+C_{\epsilon_1}\int_0^t\partial_t\|v\|^2_{L^2}\\
\leq&~\epsilon_1\|v\|^2_{H^{3.5}}+C_{\epsilon_1}\|v_0\|^2_{L^{2}}+C_{\epsilon_1}\int_0^t\|v\|_{L^2}\|v_t\|_{L^2}\\
\leq& ~ \epsilon_1\|v\|^2_{H^{3.5}}+C_{\epsilon_1}\|v_0\|^2_{L^{2}}
\\& +C_{\epsilon_1}\int_0^tP(\Vert v\Vert_{H^{3.5}}, \Vert
w\Vert_{H^{5}(\Gamma_1)}, \Vert w_{t}\Vert_{H^{3}(\Gamma_1)} ) X\,ds,
   \end{split}
   \end{equation}
where $\epsilon_1\in(0,1]$ is a small
constant to be determined; the last inequality is obtained by using
\eqref{eulerr} to bound \begin{equation} \begin{split}
\|v_t\|_{H^{1}}&\lesssim \|v\|_{H^{3.5}}\|E\|_{H^{3}}\|\nabla
v\|_{H^{1}}+(\|v\|_{H^{3.5}}+\|\psi_t\|_{H^{3}})\|\nabla
v\|_{H^{1}}+\|E\|_{H^{2}}\|q\|_{H^{2}}\\ &\leq
P(\|v\|_{H^{3.5}},\|w\|_{H^{5}(\Gamma_1)},\|w_t\|_{H^{3}(\Gamma_1)}).
   \llabel{EQ08}\end{split} \end{equation} By \eqref{platestp2} and \eqref{wd}, we
have
\begin{equation}\label{platestp3} \begin{split}
&h\|D^3w_t\|^2_{L^2}+\frac{4\lambda\mu
h^3}{24(\lambda+2\mu)}\|D^3\partial_{\alpha\alpha}w\|^2_{L^2} +
\frac{4\mu
h^3}{24}\sum_{\alpha,\beta=1}^2\|D^3\partial_{\alpha\beta}w\|^2_{L^2}\\&\indeq
\leq\|D^3w_t(0)\|+\epsilon_1\|v\|^2_{H^{3.5}}+C_{\epsilon_1}\|v_0\|^2_{L^{2}}+C_{\epsilon_1}\int_0^tP(\|w\|_{H^5},\|w_t\|_{H^3},\|q\|_{H^{2.5}})
X\,ds,  \end{split} \end{equation} 
with the implicit constants again independent of~$\nu$.
Using \eqref{EQ129}, \eqref{EQ131},
and \eqref{EQ132} in \eqref{divcurl}
and fixing a sufficiently small $\epsilon$ so that $\epsilon\Vert
v\Vert_{H^{3}}$ can be absorbed, we conclude
\begin{align}\label{velfin} \begin{split} \Vert v\Vert_{H^{3.5}}^2
&\lec \Vert \zeta(0)\Vert_{H^{2.5}}^2 + \Vert
w_{t}\Vert_{H^{3}(\Gamma_1)}^2 +
\epsilon^2\|w\|^2_{H^{5}(\Gamma_1)}\|v\|^2_{H^{3.5}}\\
&\indeq+\int_0^tP(\Vert v\Vert_{H^{3.5}}, \Vert
w\Vert_{H^{5}(\Gamma_1)}, \Vert w_{t}\Vert_{H^{3}(\Gamma_1)} ) X\,ds .
\end{split} \end{align} Multiplying \eqref{velfin} with a small
constant $\epsilon_0\in (0, 1]$ and adding the resulting inequality to
\eqref{platestp3}, we obtain
  \begin{equation}\llabel{platestp1}
  \begin{split} \epsilon_0&\Vert
  v\Vert_{H^{3.5}}^2+h\|D^3w_t\|^2_{L^2}+\frac{4\lambda\mu
  h^3}{24(\lambda+2\mu)}\|D^3\partial_{\alpha\alpha}w\|^2_{L^2}
  + \frac{4\mu
  h^3}{24}\sum_{\alpha,\beta=1}^2\|D^3\partial_{\alpha\beta}w\|^2_{L^2}
  \\&\indeq \lec
  \|D^3w_t(0)\|_{L^2(\Gamma_1)}+\epsilon_1\|v\|^2_{H^{3.5}}+C_{\epsilon_1}\|v_0\|^2_{L^{2}}+\epsilon_0\Vert
  v(0)\Vert_{H^{3.5}}^2 + \epsilon_0 \Vert
  w_{t}\Vert_{H^{3}(\Gamma_1)}^2 \\&\indeq\indeq +
  \epsilon_0\epsilon^2\|w\|^2_{H^{5}(\Gamma_1)}\|v\|^2_{H^{3.5}}+C_{\epsilon_1}\int_0^tP(\|w\|_{H^5(\Gamma_1)},\|w_t\|_{H^3(\Gamma_1)},\|q\|_{H^{2.5}})
  X \,ds.
  \end{split}
  \end{equation} 
Applying \eqref{EQ03} and using the Poincar\'e inequality, we then get
  \begin{equation}
  \begin{split} \epsilon_0&
  \Vert
  v\Vert_{H^{3.5}}^2
  +h\|w_t\|^2_{H^{3}(\Gamma_1)}
  +\frac{\lambda\mu
  h^3}{C(\lambda+2\mu)}
  \|w\|^2_{H^{5}(\Gamma_1)}
  + \frac{\mu
  h^3}{C}
  \|w\|^2_{H^{5}(\Gamma_1)}  
  \\&\indeq \lec
  \|D^3w_t(0)\|_{L^2(\Gamma_1)}+\epsilon_1\|v\|^2_{H^{3.5}}+C_{\epsilon_1}\|v_0\|^2_{L^{2}}+\epsilon_0\Vert
  v(0)\Vert_{H^{3.5}}^2 + \epsilon_0 \Vert
  w_{t}\Vert_{H^{3}(\Gamma_1)}^2 \\&\indeq\indeq +
  \epsilon_0\epsilon^2\|w\|^2_{H^{5}(\Gamma_1)}\|v\|^2_{H^{3.5}}+C_{\epsilon_1}\int_0^tP(\|w\|_{H^5(\Gamma_1)},\|w_t\|_{H^3(\Gamma_1)},\|q\|_{H^{2.5}})
  X \,ds.
  \end{split}
   \label{EQ57}
  \end{equation} 
First, we let $\epsilon_0$ be
sufficiently small to absorb the fifth term on the right-hand side
into the second term on the left-hand side. Then, we choose $\epsilon$
sufficiently small to absorb the sixth term on the right-hand side
into the first and the fourth term on the left-hand side. Lastly, we choose
$\epsilon_1$ sufficiently small to absorb the second term on the right
hand side to bound
  \begin{equation}\label{platestp4} \begin{split}
   &
  \epsilon_0
  \Vert v\Vert_{H^{3.5}}^2
  +h\|w_t\|^2_{H^{3}(\Gamma_1)}
  +\frac{\lambda\mu  h^3}{C(\lambda+2\mu)}
  \|w\|^2_{H^{5}(\Gamma_1)}
  + \frac{\mu h^3}{C}
  \|w\|^2_{H^{5}(\Gamma_1)}  
   \\&\indeq
    \lec\|w_t(0)\|_{H^3(\Gamma_1)}^2
    +\Vert v(0)\Vert_{H^{3.5}}^2
    +\int_0^tP(\|w\|_{H^5(\Gamma_1)},\|w_t\|_{H^3(\Gamma_1)}) X\,ds
   ,
  \end{split}
  \end{equation}
where we also used~\eqref{EQ78}.
Combining \eqref{EQ107} and \eqref{platestp4}, we get
  \begin{equation}
  \begin{split}
   &
  \epsilon_0
  \Vert v\Vert_{H^{3.5}}^2
  +h\|w_t\|^2_{H^{3}}
  +\frac{\lambda\mu  h^3}{C(\lambda+2\mu)}  \|w\|^2_{H^{5}(\Gamma_1)}
  + \frac{\mu h^3}{C}
  + \|w\|^2_{H^{5}(\Gamma_1)}
  + X
   \\&\indeq
    \lec\|w_t(0)\|_{H^3}^2
    +\Vert v(0)\Vert_{H^{3.5}}^2
    + X(0)
    +\int_0^tP(\|w\|_{H^5(\Gamma_1)},\|w_t\|_{H^3(\Gamma_1)}, X)\,ds
    .
  \end{split}
   \label{EQ58}
  \end{equation}
Applying a
Gronwall argument on \eqref{EQ107} and \eqref{platestp4}
(see~\cite{Bi}), we obtain an
estimate \begin{equation} \Vert v\Vert_{H^{3.5}} + \Vert
w\Vert_{H^{5}(\Gamma_1)} + \Vert w_{t}\Vert_{H^{3}(\Gamma_1)} +
X^{1/2} \lec M \llabel{EQ141} \end{equation} on $[0,T_0]$,
where $T_0$ depends on the size of initial data,
and the
proof of the theorem is concluded.  \end{proof}

\begin{Remark}
{\rm
We can generalize the estimate \eqref{platestp4} to fractional Sobolev spaces. A brief outline of the argument is provided here.
Testing \eqref{plate} with $\displaystyle \Lambda^{2(2+\delta)}w_t$, with $\delta\in (0,1.5]$, we obtain
\begin{equation}\label{plambda}
h(\Lambda^{2+\delta} w_{tt},\Lambda^{2+\delta} w_t)+(L_Kw,\Lambda^{2(2+\delta)}w_t)=(q,\Lambda^{2(2+\delta)}w_t)
.
\end{equation}
We now take a closer look at the term $(L_Kw,\Lambda^{2(2+\delta)}w_t)$. Using \eqref{elastics} and \eqref{elastico}, we write
 \begin{equation}\label{mm}
 	\begin{split}
   		(L_Kw,\Lambda^{2(2+\delta)}w_t)&=\frac{h}{2}\int_{\Gamma_1}a^{\alpha\beta\sigma\tau}G_{\alpha\beta}(w)(G'(w)\Lambda^{2(2+\delta)}w_t)_{\sigma\tau}
  		\\&\indeq +
  		\frac{h^3}{24}\int_{\Gamma_1}a^{\alpha\beta\sigma\tau}R_{\alpha\beta}(w)(R'(w)\Lambda^{2(2+\delta)}w_t)_{\sigma\tau}\\
		&=M_1+M_2.
 \end{split}
 \end{equation}
We now bound each term on the right-hand side.
For the first term, we have
 \begin{equation}\llabel{m1aa}
	\begin{split}
		M_1&
		= \frac{a^{\alpha\beta\sigma\tau}h}{2}\int_{\Gamma_1}\partial_\alpha w\partial_\beta w(\partial_\sigma \Lambda^{2(2+\delta)}w_t)\partial_\tau w+\frac{a^{\alpha\beta\sigma\tau}h}{2}\int_{\Gamma_1}\partial_\alpha w\partial_\beta w\partial_\sigma w\partial_\tau \Lambda^{2(2+\delta)}w_t
		\\&\lesssim h\int_{\Gamma_1}\left|\Lambda^{2+\delta}\left(\partial^2_{\alpha\sigma} w\partial_\beta w\partial_\tau w\right)\right|\left|\Lambda^{2+\delta}w_t\right|
		\\&\lesssim h\int_{\Gamma_1}\left|\Lambda^{2+\delta}\left(\partial^2_{\alpha\sigma} w\partial_\beta w\partial_\tau w\right)-\left(\Lambda^{2+\delta}\partial^2_{\alpha\sigma} w\right)\partial_\beta w\partial_\tau w\right|\left|\Lambda^{2+\delta}w_t\right|\\
		&\indeq +h\int_{\Gamma_1}\left|\left(\Lambda^{2+\delta}\partial^2_{\alpha\sigma} w\right)\partial_\beta w\partial_\tau w\right|\left|\Lambda^{2+\delta}w_t\right|
,
	\end{split}
\end{equation}
from where, using the Kato-Ponce inequality to estimate the commutator terms,
 \begin{equation}\label{m1}
	\begin{split}
		M_1&
              \lec
		h\|\Lambda^{2+\delta}\left(\partial_\beta w\partial_\tau w\right)\|_{L^2(\Gamma_1)}\|w\|_{H^{4+\delta}(\Gamma_1)}\|w_t\|_{H^{2+\delta}(\Gamma_1)}+P(\|w\|_{H^{4+\delta}(\Gamma_1)},\|w_t\|_{H^{2+\delta}(\Gamma_1)})
		\\&\lesssim h\|\Lambda^{2+\delta}(\partial_\beta w\partial_\tau w)-(\Lambda^{2+\delta}\partial_\beta w)\partial_\tau w\|_{L^2(\Gamma_1)}\|w\|_{H^{4+\delta}(\Gamma_1)}\|w_t\|_{H^{2+\delta}(\Gamma_1)}
		\\&\indeq +h\|(\Lambda^{2+\delta}\partial_\beta w)\partial_\tau w\|_{L^2(\Gamma_1)}\|w\|_{H^{4+\delta}(\Gamma_1)}\|w_t\|_{H^{2+\delta}(\Gamma_1)}
   \\&\indeq
+P(\|w\|_{H^{4+\delta}(\Gamma_1)},\|w_t\|_{H^{2+\delta}(\Gamma_1)})\\
&\leq P(\|w\|_{H^{4+\delta}(\Gamma_1)},\|w_t\|_{H^{2+\delta}(\Gamma_1)})
.
	\end{split}
\end{equation}
\eold
For $M_2$, we may proceed exactly as in \eqref{r}, replacing $\displaystyle -\partial^{2\gamma}$ with $\Lambda^{2(2+\delta)}$ to obtain

\begin{equation}\label{m2}
M_2=\frac{1}{2}\frac{d}{dt}\left(\frac{4\lambda\mu h^3}{24(\lambda+2\mu)}\|\Lambda^{2+\delta}\partial_{\alpha\alpha}w\|^2_{L^2(\Gamma_1)} + \frac{4\mu h^3}{24}\sum_{\alpha,\beta=1}^2\|\Lambda^{2+\delta}\partial_{\alpha\beta}w\|^2_{L^2(\Gamma_1)}\right).
\end{equation}
Inserting \eqref{mm}, \eqref{m1}, and \eqref{m2} into \eqref{plambda}, we get
\begin{equation}\llabel{317analog}
\begin{split}
  &
  \frac{1}{2}\frac{d}{dt}\Bigg(h\|\Lambda^{2+\delta}w_t\|^2_{L^2(\Gamma_1)}
     +\frac{4\lambda\mu h^3}{24(\lambda+2\mu)}\|\Lambda^{2+\delta}\partial_{\alpha\alpha}w\|^2_{L^2(\Gamma_1)}
     + \frac{4\mu h^3}{24}\sum_{\alpha,\beta=1}^2\|\Lambda^{2+\delta}\partial_{\alpha\beta}w\|^2_{L^2(\Gamma_1)}\Bigg)
     \\&\indeq
     \leq(\Lambda^{2(2+\delta)}q,w_t)+P(\|w\|_{H^{4+\delta}(\Gamma_1)},\|w_t\|_{H^{2+\delta}(\Gamma_1)})
     ,
\end{split}
   \llabel{EQ11}
\end{equation}
which is analogous to the equation~(3.17) in~\cite{KT1}. Proceeding exactly as in \cite{KT1} from this point on, we obtain
\begin{equation} 
\Vert v\Vert_{H^{2.5+\delta}} + \Vert
w\Vert_{H^{4+\delta}(\Gamma_1)} + \Vert w_{t}\Vert_{H^{2+\delta}(\Gamma_1)} +
X^{1/2} \lec M 
   \llabel{EQ15}
\end{equation} 
on $[0,T_0]$, where
\begin{equation} X = \int_{\Omega_0}
   \llabel{EQ16}|\Lambda_3^{1.5+\delta}\theta|^2 \end{equation} and $\Omega_0=\mathbb{T}^2\times\mathbb{R}$.
}
\end{Remark}

\section{Estimates with normalized curvature tensor}
\label{sec04}
In this section, we consider the normalized version of the tensor $R$ that was defined in Section~\ref{elastic} and is given by
\begin{equation}\label{curvex}
R_{\alpha\beta}=\frac{1}{|a_1(w)\times a_2(w)|}\partial_\alpha a_\beta(w)\cdot n(w)-\partial_\alpha a_\beta\cdot n
,
\end{equation}
and this in fact is the actual change in the second fundamental form.

We now state a~priori estimates for the plate with the exact change in the curvature tensor. 

\cole
\begin{Lemma}\label{EstimatePlate}
Under the assumptions of Theorem~\ref{main} and with $R_{\alpha\beta}$ defined in \eqref{curvex}, we have
  \begin{equation}
  \begin{split}
  &h\|D^3w_t\|^2_{L^2(\Gamma_1)}
    +\frac{\lambda\mu h^3}{6(\lambda+2\mu)}\|D^3\partial_{\alpha\alpha}w\|^2_{L^2(\Gamma_1)}  + \frac{\mu h^3}{6}\sum_{\alpha,\beta=1}^2\|D^3\partial_{\alpha\beta}w\|^2_{L^2(\Gamma_1)}
  +\nu \|D^4 w_t\|_{L^2(\Gamma_1)}^2
  \\&\indeq
   \lec\|D^3w_t(0)\|_{L^2(\Gamma_1)}+\|v\|_{L^2}^{2/7}\|v\|_{H^{3.5}}^{12/7}+\int_0^tP(\|w\|_{H^5(\Gamma_1)},\|w_t\|_{H^3(\Gamma_1)},\|q\|_{H^{2.5}}).
\end{split}
   \llabel{EQ65}
   \end{equation}
   \end{Lemma}
\colb

\begin{proof}
For our geometry,
\begin{equation}\llabel{modr}
	R_{\alpha\beta}=\frac{1}{|a_1(w)\times a_2(w)|}\partial_\alpha a_\beta(w)\cdot n(w)-\partial_\alpha a_\beta\cdot n=\frac{\partial^2_{\alpha\beta}w}{\sqrt{1+(\partial_1w)^2+(\partial_2w)^2}}.
\end{equation}
The action of the Fr\'echet derivative of $R$ is given by
\begin{equation}\label{modrp}
	\begin{split}
		\left(R'(w)\xi\right)_{\alpha\beta}&=\left(\lim_{h\to0}\frac{R(w+h\xi)-R(w)}{h}\right)_{\alpha\beta}\\
		&=-\frac{\partial_1w\partial_1\xi+\partial_2w\partial_2\xi}{(1+(\partial_1w)^2+(\partial_2w)^2)^{3/2}}\partial^2_{\alpha\beta}w+\frac{\partial^2_{\alpha\beta}\xi}{\sqrt{1+(\partial_1w)^2+(\partial_2w)^2}}
      .
	\end{split}
\end{equation}
Let $\gamma, \gamma_1, \gamma_2\in \mathbb{N}_0^2$ be such that
$|\gamma_1|=1$, $|\gamma_2|=2$, and 
$|\gamma|=3$, with $\partial^{\gamma}=\partial^{\gamma_1}\partial^{\gamma_2}$. This  then leads to the following modifications to~\eqref{r}:
\begin{equation}\label{L}
  \begin{split}
   &
    -\frac{h^3a^{\alpha\beta\sigma\tau}}{24}\int_{\Gamma_1}R_{\alpha\beta}(w)\left[R'(w)\partial^{2\gamma}w_t\right]_{\sigma\tau}
   \\&\indeq
  =\frac{h^3a^{\alpha\beta\sigma\tau}}{24}\int_{\Gamma_1}
	\frac{\partial_1w\partial_1\partial^{2\gamma}w_t+\partial_2w\partial_2\partial^{2\gamma}w_t}{(1+(\partial_1w)^2+(\partial_2w)^2)^{2}}\partial^2_{\alpha\beta}w\partial^2_{\sigma\tau}w
   \\&\indeq\indeq
	\indeq-\frac{h^3a^{\alpha\beta\sigma\tau}}{24}\int_{\Gamma_1}\frac{\partial^2_{\alpha\beta}w\partial^2_{\sigma\tau}\partial^{2\gamma}w_t}{1+(\partial_1w)^2+(\partial_2w)^2}
      \\&\indeq
	= I_1+I_2.
	\end{split}
\end{equation}
We now individually handle the two terms on the right-hand side of~\eqref{L}. For the first term on the right-hand side, we have
\begin{equation}\label{4.5}
	\begin{split}
	I_1&=\frac{h^3a^{\alpha\beta\sigma\tau}}{24}\int_{\Gamma_1}\frac{\partial_1w\partial_1\partial^{2\gamma}w_t+\partial_2w\partial_2\partial^{2\gamma}w_t}{(1+(\partial_1w)^2+(\partial_2w)^2)^{2}}\partial^2_{\alpha\beta}w\partial^2_{\sigma\tau}w\\
	&\leq \frac{h^3}{24}\|D^2w\|_{L^\infty(\Gamma_1)}^2\int_{\Gamma_1}(|D^4\partial_1w| |D^2\partial_1 w_t| + |D^4\partial_2w| |D^2\partial_2 w_t|)\\
	&\lec {h^3}\|D^5w\|_{L^2(\Gamma_1)}^2\|D^3w_t\|_{L^2(\Gamma_1)}
        ,
	\end{split}
\end{equation}
while the second term on the right-hand side of \eqref{L} may be rewritten as
\begin{equation}\label{sec}
	\begin{split}
		I_2&=
\frac{h^3a^{\alpha\beta\sigma\tau}}{24}\int_{\Gamma_1}\frac{\partial^2_{\alpha\beta}\partial^\gamma w\partial^2_{\sigma\tau}\partial^\gamma w_t}{1+(\partial_1w)^2+(\partial_2w)^2}
   \\&\indeq
   +\frac{h^3a^{\alpha\beta\sigma\tau}}{24}\int_{\Gamma_1}\partial^2_{\alpha\beta}\partial^{\gamma_2}w\partial^2_{\sigma\tau}D^3w_t\partial^{\gamma_1}(1+(\partial_1w)^2+(\partial_2w)^2)^{-1}
     \\&\indeq
     +\frac{h^3a^{\alpha\beta\sigma\tau}}{24}\int_{\Gamma_1}\partial^2_{\alpha\beta}\partial^{\gamma_1}w\partial^2_{\sigma\tau}\partial^\gamma w_t\partial^{\gamma_2}(1+(\partial_1w)^2+(\partial_2w)^2)^{-1}
     \\&\indeq
     +\frac{h^3a^{\alpha\beta\sigma\tau}}{24}\int_{\Gamma_1}\partial^2_{\alpha\beta}w\partial^2_{\sigma\tau}\partial^\gamma w_t\partial^\gamma(1+(\partial_1w)^2+(\partial_2w)^2)^{-1}\\
     &=I_{21}+I_{22}+I_{23}+I_{24}
     .
	\end{split}
\end{equation}
We now bound the last three terms on the right-hand side of~\eqref{sec}. Let $\partial^{\gamma_2}=\partial^{\gamma_3}\partial^{\gamma_4}$, with $|\gamma_3|=|\gamma_4|=1$. For the second term, we integrate by parts multiple times to obtain
\begin{equation}\llabel{4.10}
\begin{split}
	I_{22}&=\frac{h^3a^{\alpha\beta\sigma\tau}}{24}\int_{\Gamma_1}\partial^2_{\alpha\beta}\partial^{\gamma_2}w\partial^2_{\sigma\tau}D^3w_t\partial^{\gamma_1}(1+(\partial_1w)^2+(\partial_2w)^2)^{-1}\\
	&=-\frac{h^3a^{\alpha\beta\sigma\tau}}{24}\int_{\Gamma_1}\partial^2_{\alpha\beta}\partial^{\gamma_3}w \partial^2_{\sigma\tau}\partial^{\gamma_4}\partial^\gamma w_t\partial^{\gamma_1}(1+(\partial_1w)^2+(\partial_2w)^2)^{-1}\\
	&\indeq-\frac{h^3a^{\alpha\beta\sigma\tau}}{24}\int_{\Gamma_1}\partial^2_{\alpha\beta}\partial^{\gamma_3}w\partial^2_{\sigma\tau}\partial^\gamma w_t\partial^{\gamma_4}\partial^{\gamma_1}(1+(\partial_1w)^2+(\partial_2w)^2)^{-1}\\
	&\lec\|\partial^2Dw\|_{L^\infty(\Gamma_1)}\|\partial^2Dw_t\|_{L^2(\Gamma_1)}\|D^4(1+(\partial_1w)^2+(\partial_2w)^2)^{-1}\|_{L^2(\Gamma_1)}\\
	&\leq P(\|D^5w\|_{L^2(\Gamma_1)},\|D^3w_t\|_{L^2(\Gamma_1)}).
	\end{split}
\end{equation}
The third term on the right-hand side of \eqref{sec} may be bounded using integrating by parts as
\begin{equation}\llabel{4.11}
\begin{split}
I_{23}&=\frac{h^3a^{\alpha\beta\sigma\tau}}{24}\int_{\Gamma_1}\partial^2_{\alpha\beta}\partial^{\gamma_1}w\partial^2_{\sigma\tau}\partial^\gamma w_t\partial^{\gamma_2}(1+(\partial_1w)^2+(\partial_2w)^2)^{-1}\\
&\lec \|\partial^2_{\alpha\beta}Dw\|_{L^\infty(\Gamma_1)}\|\partial^2_{\sigma\tau}Dw_t\|_{L^2(\Gamma_1)}\|D^4(1+(\partial_1w)^2+(\partial_2w)^2)^{-1}\|_{L^2(\Gamma_1)}\\
&\leq P(\|D^5w\|_{L^2(\Gamma_1)},\|D^3w_t\|_{L^2(\Gamma_1)}).
\end{split}
\end{equation}
Lastly, for the fourth term, integrating by parts again multiple times, we get
\begin{equation}\llabel{4.12}
\begin{split}
I_{24}&=\frac{h^3a^{\alpha\beta\sigma\tau}}{24}\int_{\Gamma_1}\partial^2_{\alpha\beta}w\partial^2_{\sigma\tau}\partial^\gamma w_t\partial^\gamma(1+(\partial_1w)^2+(\partial_2w)^2)^{-1}\\
&=-\frac{h^3a^{\alpha\beta\sigma\tau}}{24}\int_{\Gamma_1}\partial^2_{\alpha\beta}w\partial^2_{\sigma\tau}\partial^{\gamma_1}\partial^\gamma w_t\partial^{\gamma_2}(1+(\partial_1w)^2+(\partial_2w)^2)^{-1}\\
&\indeq-\frac{h^3a^{\alpha\beta\sigma\tau}}{24}\int_{\Gamma_1}\partial^2_{\alpha\beta}\partial^{\gamma_1}w\partial^2_{\sigma\tau}\partial^\gamma w_t\partial^{\gamma_2}(1+(\partial_1w)^2+(\partial_2w)^2)^{-1}\\
&\lec  \|\partial^2_{\alpha\beta}D^3w\|_{L^2(\Gamma_1)}\|\partial^2_{\sigma\tau}Dw_t\|_{L^2(\Gamma_1)}\|D^2(1+(\partial_1w)^2+(\partial_2w)^2)^{-1}\|_{L^\infty(\Gamma_1)}\\
&\leq P(\|D^5w\|_{L^2(\Gamma_1)},\|D^3w_t\|_{L^2(\Gamma_1)}).
\end{split}
\end{equation}
We rewrite the first term on the right-hand side of \eqref{sec} as
\begin{equation}\llabel{eq413}
\begin{split}
	I_{21}&=\frac{h^3a^{\alpha\beta\sigma\tau}}{24}\int_{\Gamma_1}\frac{\partial^2_{\alpha\beta}\partial^\gamma w\partial^2_{\sigma\tau}\partial^\gamma w_t}{1+(\partial_1w)^2+(\partial_2w)^2}\\
	&=\frac{h^3\lambda\mu}{6(\lambda+2\mu)}\int_{\Gamma_1}\frac{\partial^2_{\alpha\alpha}\partial^\gamma w\partial^2_{\sigma\sigma}\partial^\gamma w_t}{1+(\partial_1w)^2+(\partial_2w)^2} +\frac{\mu h^3}{6}\int_{\Gamma_1}\frac{\partial^2_{\alpha\beta}\partial^\gamma w\partial^2_{\alpha\beta}\partial^\gamma w_t}{1+(\partial_1w)^2+(\partial_2w)^2}\\
	&=\frac{h^3\lambda\mu}{12(\lambda+2\mu)}\frac{d}{dt}\int_{\Gamma_1}\frac{(\partial^2_{\alpha\alpha}\partial^\gamma w)^2}{1+(\partial_1w)^2+(\partial_2w)^2}+ \frac{\mu h^3}{12}\frac{d}{dt}\int_{\Gamma_1}\frac{(\partial^2_{\alpha\beta}\partial^\gamma w)^2}{1+(\partial_1w)^2+(\partial_2w)^2}
	\\&\indeq
	+\frac{h^3\lambda\mu}{6(\lambda+2\mu)}\int_{\Gamma_1}\frac{(\partial^2_{\alpha\alpha}\partial^\gamma w)^2(\partial_1w\partial_1w_t+\partial_2w\partial_2w_t)}{(1+(\partial_1w)^2+(\partial_2w)^2)^2}
	\\&\indeq
	+\frac{\mu h^3}{6}\int_{\Gamma_1}\frac{(\partial^2_{\alpha\beta}\partial^\gamma w)^2(\partial_1w\partial_1w_t+\partial_2w\partial_2w_t)}{(1+(\partial_1w)^2+(\partial_2w)^2)^2}\\
	&=I_{211}+I_{212}+I_{213}+I_{214}.
\end{split}
\end{equation}
The terms $I_{211}$ and $I_{212}$ are coercive.
The last two terms on the right-hand side are bounded as
\begin{equation}\label{eq414}
\begin{split}
	I_{213}+I_{214}&=\frac{h^3\lambda\mu}{6(\lambda+2\mu)}\int_{\Gamma_1}\frac{(\partial^2_{\alpha\alpha}\partial^\gamma w)^2(\partial_1w\partial_1w_t+\partial_2w\partial_2w_t)}{(1+(\partial_1w)^2+(\partial_2w)^2)^2}\\
	&\indeq+\frac{\mu h^3}{6}\int_{\Gamma_1}\frac{(\partial^2_{\alpha\beta}\partial^\gamma w)^2(\partial_1w\partial_1w_t+\partial_2w\partial_2w_t)}{(1+(\partial_1w)^2+(\partial_2w)^2)^2}\\
	&\leq \left(\frac{h^3\lambda\mu}{6(\lambda+2\mu)}\|\partial^2_{\alpha\alpha}\partial^\gamma w\|_{L^2(\Gamma_1)}+\frac{\mu h^3}{6}\|\partial^2_{\alpha\beta}\partial^\gamma w\|_{L^2(\Gamma_1)}\right)
\\&
\indeq\indeq\indeq\indeq\indeq\indeq\indeq\indeq\indeq\indeq\indeq\indeq\indeq\indeq\times
\|\partial_1w\partial_1w_t+\partial_2w\partial_2w_t\|_{L^\infty(\Gamma_1)}\\
	&\leq P(\|w\|_{H^5(\Gamma_1)},\|w_t\|_{H^3(\Gamma_1)}).
	\end{split}
\end{equation}
We can now obtain a~priori estimates for the plate by inserting  \eqref{d6}, \eqref{g}, \eqref{L}, \eqref{4.5}, \eqref{sec}--\eqref{eq414} into \eqref{plateen}, giving
	\begin{equation}\label{4.15}
		\begin{split}
		&\frac{1}{2}\frac{d}{dt}\Bigg[h\|D^3w_t\|^2_{L^2(\Gamma_1)}+\frac{\lambda\mu h^3}{6(\lambda+2\mu)}\left\|\frac{D^3\partial_{\alpha\alpha}w}{|n(w)|}\right\|_{L^2(\Gamma_1)}  +	 \frac{\mu h^3}{6}\left\|\frac{D^3\partial_{\alpha\beta}w}{|n(w)|}\right\|^2_{L^2(\Gamma_1)}\Bigg]
        \\&\indeq\indeq
		+\nu \|D^4 w_t\|_{L^2(\Gamma_1)}^2
		\\&\indeq
		\leq-(D^6q,w_t)_{L^2(\Gamma_1)}+P(\|w\|_{H^5(\Gamma_1)},\|w_t\|_{H^3(\Gamma_1)}),
		\end{split}
	\end{equation}
where $|n(w)|=\sqrt{1+(\partial_1w)^2+(\partial_2w)^2}$ is the Euclidean norm of the dynamic normal. Integrating \eqref{4.15} in time and repeating the steps in Section~\ref{platesec} to eliminate the higher-order term containing pressure (details in \cite[Lemma 3.2]{KT1}), we get
	\begin{equation}\label{platestp3a}
		\begin{split}
			&h\|D^3w_t\|^2_{L^2(\Gamma_1)}+\frac{\lambda\mu h^3}{6(\lambda+2\mu)}\left\|\frac{D^3\partial_{\alpha\alpha}w}{|n(w)|}\right\|_{L^2(\Gamma_1)}  +	 \frac{\mu h^3}{6}\left\|\frac{D^3\partial_{\alpha\beta}w}{|n(w)|}\right\|^2_{L^2(\Gamma_1)}\\
			&\lesssim\|w_t(0)\|_{H^3(\Gamma_1)}^2+\|v\|_{L^2}^{2/7}\|v\|_{H^{3.5}(\Gamma_1)}^{12/7}+\int_0^tP(\|w\|_{H^5(\Gamma_1)},\|w_t\|_{H^3},\|q\|_{H^{2.5}(\Gamma_1)}).
		\end{split}
	\end{equation}
Now, using \eqref{fdef}, \eqref{psi}, and \eqref{matest}, we may write
\begin{equation}\label{final}
	|n(w)|\lesssim 1+\|\nabla w\|_{L^{\infty}(\Gamma_1)} \lesssim 1+\|w\|_{H^3(\Gamma_1)}\lesssim 1+\|\psi\|_{H^{3.5}}
\lec  1
.
\end{equation}
Combining \eqref{platestp3a} and \eqref{final}, we obtain the statement of the lemma.
\end{proof}

\begin{Remark}
{\rm
Note that the a~priori plate estimates for the exact curvature tensor in \eqref{platestp3a} are exactly the same as the one obtained in~\eqref{platestp2}. Hence, Theorem~\ref{main} holds for the exact curvature tensor case as well.
}
\end{Remark}

\section{Construction of solutions}

To construct solutions, we employ a fixed point technique to solve the Euler equation \eqref{eulerr} with given variable coefficient matrix $E$ and given non-homogeneous boundary data $w_t$ in the boundary condition~\eqref{EQ21}.

\subsection{Euler equations with given coefficients and non-homogeneous boundary data}

Let $T>0$. Suppose $v_{0} \in H^{3.5}(\Omega)$, and assume that
$\tilde{w}$ is a given function on $[0,T] \times \Gamma_{1}$
such that
the following assumptions hold.

\begin{assumption}\label{WA}
{\rm
1.  The function $\tilde{w}$ and its time derivatives satisfy
  \begin{equation}
   \partial^{j}_{t} \tilde{w} \in L^{\infty}([0,T]; H^{5-2j}(\Gamma_{1}))
  \comma j=0,1,2.
   \llabel{EQ64}
  \end{equation}
2. At the time $t=0$, we have $\tilde{w}=0$.\\
3. The time derivative $\tilde{w_{t}}$ satisfies the zero mean condition $\int_{\Gamma_{1}} \tilde{w}_{t}  =0$.\\
4. The divergence-free condition $\div v_{0}=0$ holds.\\
5. The impermeability condition $(v_{0})_{3} |_{\Gamma_{0}} =0$ holds.\\
6. The compatibility condition $\tilde{w}_{t}(0,\cdot)=(v_{0})_{3} |_{\Gamma_{1}}$ holds.
}
\end{assumption}

Next, let $\tilde{\psi}$ be the harmonic extension of $1+\tilde{w}$ to the domain $\Omega$, i.e.,
\begin{align}
   \begin{split}
   &\Delta \tilde{\psi} = 0
     \inon{on $\Omega$}
   \\&
   \tilde{\psi}(x_1,x_2,1,t)=1+\tilde{w}(x_1,x_2,t)
     \inon{on $\Gamma_1\times [0,T]$}
   \\&
   \tilde{\psi}(x_1,x_2,0,t)=0
     \inon{on $\Gamma_0\times [0,T]$}
   .
   \end{split}
   \label{EQ189}
  \end{align}
By the above assumptions, it follows that $\psi(0,x) =x_{3}$, $\tilde{E}(0)=I$, and $ \nabla \partial_{t} \tilde{\psi} \in C^{1}([0,T]; C^{1}(\Omega))$.

Furthermore, we assume the following.

\begin{assumption}\label{det}
{\rm For all $t \in [0,T]$, there exists a constant $c >0$ such that $\partial_{3} \tilde{\psi} \geq c $.
}
\end{assumption}

The last assumption can be realized by choosing $T$ sufficiently small given that $\partial_{3} \tilde{\psi} =1 $ at $t=0$.

We consider the system of equations
  \begin{align}
   \begin{split}
   &
    \partial_{t} v_i
    + v_1 \tilde{E}_{j1} \partial_{j} v_i
    + v_2 \tilde{E}_{j2} \partial_{j} v_i
    + (v_3-\psi_t)\tilde{E}_{33} \partial_{3} v_i
    + \tilde{E}_{ki}\partial_{k}q=0
    \\&
    \tilde{E}_{ki} \partial_{k}v_i=0 
    ,
   \end{split}
   \label{EQ188}
  \end{align}
with the boundary conditions described below.
  
Here, $\tilde{E}$ is 
defined  as the inverse of the matrix $\nabla \tilde{\eta}$ where $\tilde{\eta}= (x_{1}, x_{2}, \tilde{\psi})$ and consequently
 \begin{align}
  \begin{split}
   \nabla   \tilde{\eta}=
        \left( \begin{matrix} 1 & 0 & 0 \\
          0 & 1& 0\\
         -\fractext{\partial_{1} \tilde{\psi}} {\partial_{3} \tilde{\psi}}  
            &       -\fractext{\partial_{2} \tilde{\psi}} {\partial_{3} \tilde{\psi}}  
            & \fractext{1} {\partial_{3} \tilde{\psi}}\end{matrix}\right)  
     .
  \end{split}
  \label{EQq190}
  \end{align}
By  our assumptions, it follows that the matrix $\nabla \tilde{\eta}$ is non-singular 
on $[0,T]$ with a
well-defined inverse~$\tilde{E}$, i.e., $\partial_{3} \tilde{\psi} \neq 0$.
Thus the matrix $\tilde{E}(\tilde{w})$ depends on the given data $\tilde{w}$ and has the explicit expression  
  \begin{align}
  \begin{split}
      \tilde{E}=
        \left( \begin{matrix} 1 & 0 & 0 \\
          0 & 1& 0\\
         -\fractext{\partial_{1} \tilde{\psi}} {\partial_{3} \tilde{\psi} }  
            &       -\fractext{\partial_{2} \tilde{\psi} } {\partial_{3} \tilde{\psi} }  
            & \fractext{1} {\partial_{3} \tilde{\psi} }\end{matrix}\right)  
    .
  \end{split}
  \label{EQ190}
  \end{align}
We again denote by $\tilde{F}$ the corresponding  cofactor matrix of $\tilde{E}$, or explicitly
  \begin{align}
  \begin{split}
   \tilde{F}
      = (\partial_{3} \tilde{\psi})\tilde{a}
      =
      \begin{pmatrix}
       \partial_{3}\tilde{\psi}&  0 & 0 \\
       0 & \partial_{3}\tilde{\psi}& 0 \\
       -\partial_{1} \tilde{\psi} & -\partial_{2}\tilde{\psi} & 1
     \end{pmatrix}
     ,
  \end{split}
     \label{EQ191}
   \end{align}
which again satisfies the Piola identity~\eqref{piola}.
Moreover, we impose the following assumption.

\begin{assumption}\label{small}
{\rm
We have
  \begin{align}
   \Vert \tilde{E} -I \Vert_{L^{\infty}([0,T];H^{4.5}(\Omega))} 
   \leq \epsilon 
   \label{EQ197}
  \end{align} 
and 
  \begin{align}
   \Vert \tilde{F} -I \Vert_{L^{\infty}([0,T];H^{4.5}(\Omega))} \leq \epsilon 
   ,
   \label{EQ198}
  \end{align} 
for some $\epsilon >0$ sufficiently small. 
}
\end{assumption}

On the bottom rigid boundary, we assume the usual impermeability condition
  \begin{equation}
   v_3=0
   \inon{on $\Gamma_0$}
   ,
   \label{EQ193}
  \end{equation}
while on the top we prescribe
  \begin{equation}
     \tilde{F}_{3i}v_i = \tilde{w}_{t}
     \inon{on $\Gamma_1$}
    .
   \label{EQ194}
  \end{equation}
Note that we have the estimate
  \begin{align}
   \Vert \tilde{E}  \Vert_{L^{\infty}([0,T];H^{s-1/2}(\Omega))} 
   \leq 
   \Vert \tilde{w}  \Vert_{L^{\infty}([0,T];H^{s}(\Gamma_{1}))} 
   ,
   \label{EQ199}
  \end{align} 
for $t\in [0,T]$ and $s >1/2$.

We now invoke the following theorem pertaining to the above Euler system with given coefficients, the proof of which can be found in~\cite{KT1}.

\cole
\begin{Theorem}
\label{T04}
Assume that $v_{0}\in H^{3.5}$ and that $\tilde{w}$ is given satisfying Assumptions~\ref{WA},  and $\tilde{\psi}$ defined in \eqref{EQ189}, while the matrix $E$ is defined by \eqref{EQ190}  and satisfying Assumptions~\ref{det} and~\ref{small} on an interval $[0,T]$.
Then, there exists a local-in-time solution $(v,q)$ to the system \eqref{EQ188} with the boundary conditions \eqref{EQ193} and \eqref{EQ194}
such that
  \begin{align}
  \begin{split}
   &v \in L^{\infty}([0,T_{0}];H^{3.5}(\Omega)) 
    \\&
      v_{t} \in L^{\infty}([0,T_{0}];H^{1.5}(\Omega)) 
    \\&
     q \in L^{\infty}([0,T_{0}];H^{2.5}(\Omega)),
  \end{split}
   \llabel{EQ201}
  \end{align} 
for a time $T_{0}\in (0,T]$ depending on the initial data.
The solution is unique up to an additive function of time for the pressure~$q$.
Moreover, the solution 
$(v,q)$
satisfies the estimate
  \begin{align}
  \begin{split}
   \Vert v(t) \Vert_{H^{3.5}}
   + \Vert \nabla q(t) \Vert_{H^{1.5}} 
   \leq
   \Vert v_{0} \Vert_{H^{3.5}} 
    + \int_{0}^{t} P(
                      \Vert \tilde{w}(s) \Vert_{H^{5}(\Gamma_{1})},
                      \Vert \tilde{w}_{t}(s) \Vert_{H^{3}(\Gamma_{1})}
        )\, ds
  ,
  \end{split}
   \label{EQ50}
  \end{align}
for $t\in[0,T_{0})$.
\end{Theorem}
\colb

\subsection{ Local existence theorem for the plate equation}

For a given forcing
and with periodic boundary conditions on $\mathbb{R}^{2}$, we invoke
the following
existence
theorem.

\cole
\begin{Theorem}
\label{T01}
Let $T>0$ and $\nu >0$. Given $w(0,\cdot)=0$,  $w_{t}(0,\cdot) \in  H^{3}(\Gamma_{1})$, and $q \in L^{2}([0,T]; H^{2}(\Gamma_{1}))$  such that $\int_{\Gamma_{1}} q =0$,
for the time $T_0$ such that
  \begin{equation}
   T_0^{-1}
   =
   P\left(\Vert w_t(0,\cdot)\Vert_{H^{3}(\Gamma_1)},\nu^{-1}\int_{0}^{T} \Vert q\Vert_{H^{2}(\Gamma_1)}^2\,ds\right)   
   \label{EQ02}
  \end{equation}
where $P$ is a polynomial,
there exists a unique solution $(w, w_{t})$ to the equation 
\begin{equation}
h w_{tt}+L_Kw - \nu \Delta  w_{t} =q \inon{ on $\Gamma_1 \times [0,T_0]$},
   \llabel{EQ60}
\end{equation} 
with periodic boundary conditions, such that 
\begin{equation}
(w, w_{t}) \in L^{\infty}([0,T_0]; H^{5}(\Gamma_{1}) \times H^{3}(\Gamma_{1})),
   \llabel{EQ59}
\end{equation}
and, for every $t\in [0,T_0]$,
  \begin{align}\label{EstPlate}
 \Vert w(t) \Vert^{2}_{H^{5}(\Gamma_{1})} + \Vert w_{t}(t) \Vert^{2}_{H^{3}(\Gamma_{1})} + \nu \int_{0}^{t} \Vert w_{t}(s) \Vert^{2}_{H^{4}(\Gamma_{1})} \, ds
   \lesssim   \Vert w_{1} \Vert^{2}_{H^{3} (\Gamma_{1})}
  .
 \end{align}
Moreover, $\int_{\Gamma_{1}} w_{t} =0$ for all $t \in [0,T_0]$.
\end{Theorem}
\colb

\begin{proof}[Proof of Theorem~\ref{T01}]
First, we apply a formal Gronwall argument, which we then justify
using the Galerkin method.
Applying 
the Cauchy-Schwarz inequality
and $\int_{0}^{t}w=\int_{0}^{t}w_t=0$
on \eqref{platestp}, we obtain
 \begin{equation}
 \begin{split}
  &
  \frac{h}{C}\|w_t\|^2_{H^{3}(\Gamma_1)}
     +\frac{\lambda\mu h^3}{C(\lambda+2\mu)}\|w\|^2_{H^{5}(\Gamma_1)}
     + \frac{\mu h^3}{C}\|w\|^2_{H^{5}(\Gamma_1)}
     + \nu \int_{0}^{t}  \Vert w_t \Vert^{2}_{H^{4}(\Gamma_1)}
     \\&\indeq
     \lec
     \Vert w_1\Vert_{H^{3}(\Gamma_1)}^2
     +
     \frac{1}{\nu}
     \int_{0}^{t}
     \Vert D^2 q\Vert_{L^2}^2
     +P\left(\|w\|_{H^5(\Gamma_1)},\|w_t\|_{H^3(\Gamma_1)}\right)
     ,
  \end{split}
   \label{EQ52}
  \end{equation}
for all~$t$,
where $P$ is twice the polynomial in~\eqref{platestp}.
Recalling that $w(0)=0$,
we deduce
  \begin{align}
  \begin{split}
   &
   \sup_{0\leq t\leq T}
   \Vert w\Vert_{H^{5}(\Gamma_1)}^2
   +
   \sup_{0\leq t\leq T}
   \Vert w_t\Vert_{H^{3}(\Gamma_1)}^2
   \lec
   \Vert w_1\Vert_{H^{3}(\Gamma_1)}^2
   +
     \frac{1}{\nu}
   \int_{0}^{T} \Vert D^2 q\Vert_{L^2(\Gamma_1)}^2
    \comma t\in[0,T]
   ,
  \end{split}
   \label{EQ53}
  \end{align}
where $T=1/P(\Vert w_1\Vert_{H^{3}})$ and $P$ is a polynomial, which depends on the polynomial
in~\eqref{EQ52}.

The proof of the theorem uses the Galerkin method.
Here, we outline the key steps of the proof. Moreover, we focus on the more challenging case of the operator $L_K$ defined in Section~\ref{sec04}; the proof for the simpler case of the operator $L_K$ defined in Section \ref{elastic} can be derived analogously.

First, we apply the Galerkin approximation to construct the approximate solution~$w^n$. By integrating \eqref{plattee} over $\Gamma_1$ and using \eqref{elastico}, \eqref{elastics}, and \eqref{modrp}, we have
\begin{align*}
	0=\int_{\Gamma_1}w^n_t=\frac{d}{dt}\int_{\Gamma_1}w^n.
\end{align*}
Thus, the Poincar\'e inequality holds for the Galerkin
approximations. Proceeding as in the proof of Lemma
\ref{EstimatePlate}, we establish that the inequality \eqref{EQ52} holds for the Galerkin approximations.
By applying Gronwall's inequality, as above, we deduce that there exists $T_0>0$
as in the statement such that estimate \eqref{EstPlate} holds uniformly for the Galerkin approximations~$w^n$. Given that we have established higher-order estimates, we can invoke the Lions-Aubin lemma to obtain the appropriate strong convergence properties of the sequence $w^n$ and deduce that the limit satisfies the plate equation and the estimate~\eqref{EQ53}.
\end{proof}

\subsection{Local existence of solutions to the coupled system via a fixed point}
We now prove the following result on existence of a local-in-time solution to the regularized (damped system with $\nu >0$) via a fixed point argument.

\cole
\begin{Theorem}\label{T2}
Let $\nu >0$. Given initial data $v_{0} \in H^{3.5}(\Omega)$, $w_{0}=0$, and $w_{1} \in H^{3}(\Gamma_{1})$  satisfying Assumption~\ref{WA}~(4) and (6), and $\int_{\Gamma_{1}} w_{1}=0$,  there exists a local-in-time solution $(v,q,w, w_{t})$  to the system \eqref{eulerr}, \eqref{plattee}, \eqref{vg0}, and \eqref{EQ21} such that
\begin{align*}
&v \in L^{\infty}([0,T']; H^{3.5}(\Omega)) \cap C([0,T'];H^{1.5}(\Omega))\\
&v_{t} \in L^{\infty}([0,T']; H^{1.5}(\Omega))\\
&q  \in L^{\infty}([0,T']; H^{2.5}(\Omega))\\
&(w, w_{t}) \in L^{\infty}([0,T']; H^{5}(\Gamma_{1}) \times H^{3}(\Gamma_{1}))
,
\end{align*}
for some time $T'>0$, depending on the initial data and~$\nu$.
\end{Theorem}
\colb

\begin{proof}[Proof of Theorem]
We construct a solution to the above system using the iteration scheme with $(v^{(n+1)}, q^{(n+1)}) $ solving the system 
 \begin{align}
  \begin{split}
   &
    \partial_{t} v^{(n+1)}_i
    + v^{(n+1)}_1 E^{(n)}_{j1} \partial_{j} v^{(n+1)}_i
    + v^{(n+1)}_2 E^{(n)}_{j2} \partial_{j} v^{(n+1)}_i
   \\&\indeq
    + (v^{(n+1)}_3-\psi^{(n)}_t) E^{(n)}_{33} \partial_{3} v^{(n+1)}_i
    + E^{(n)}_{ki}\partial_{k}q^{(n+1)}
    =0
    ,
    \\&
   E^{(n)}_{ki} \partial_{k}v^{(n+1)}_i=0 
    \comma i, j,k=1,2,3
    ,
  \end{split}
   \label{EQ250}
  \end{align}
with the boundary conditions
   \begin{align}
 &  v^{(n+1)} = 0 \inon{on $\Gamma_0$}
   \llabel{EQ07}   
    \\  
 & \tilde{F}^{(n)}_{3i}v^{(n+1)}_i = w^{(n)}_t \inon{on $\Gamma_1$}
    \comma i=1,2,3
   \llabel{EQ06}
    ,
\end{align}
given $(w^{(n)}, w_{t}^{(n)})$; the matrix $E^{(n)}$ is determined from the harmonic extension $\psi^{(n)}$ of $1+w^{(n)}$ satisfying
 \begin{align}
   \begin{split}
      &\Delta \psi^{(n)}= 0
     \onon{\Omega}
   \\
   &
   \psi^{(n)}(x_1,x_2,1,t)=1+w^{(n)} (x_1,x_2,t)
     \inon{on $\Gamma_1$}
   \\
     &
   \psi^{(n)}(x_1,x_2,0,t)=0
     \inon{on $\Gamma_0$}
   \end{split}
   \llabel{EQ251}
  \end{align}
by the expression
  \begin{align}
   E^{(n)}=
    \begin{pmatrix} 1 & 0 & 0 \\
      0 & 1& 0\\
      -\fractext{\partial_{1} \psi^{(n)}} {\partial_{3} \psi^{(n)}}  
             &     -\fractext{\partial_{2} \psi^{(n)}} {\partial_{3} \psi^{(n)}}  
             & \fractext{1} {\partial_{3} \psi^{(n)}}
    \end{pmatrix}
    ,
  \label{EQ252} 
  \end{align}
and $F^{(n)}$ is the corresponding cofactor matrix. On the other hand, given $q^{(n+1)}$,
the iterate $w^{(n+1)}$ is determined from the plate equation
 \begin{align}
   &  w^{(n+1)}_{tt} +\Delta^{2} w^{(n+1)} - \nu \Delta_{2} w^{(n+1)}_{t} 
   =q^{(n+1)}
  \inon{on $\Gamma_{1}  \times[0,T]$}
   .
   \label{EQ253} 
  \end{align}
Now, assume that $w^{(n)}=\tilde{w}$
satisfies
\begin{equation}
\sum_{j=0}^{2} \Vert \partial^{j}_{t} w^{(n)} \Vert_{H^{5-2j}} \lesssim M
   = 2C( \Vert w_{1} \Vert_{H^{3} (\Gamma_{1})} 
	+
   \Vert v_{0} \Vert_{H^{3.5}} )
,
\end{equation} 
along with Assumption~\ref{WA},
and suppose that the matrix $E^{(n)}= \tilde{E}$ is defined as in \eqref{EQ90} under Assumptions~\ref{det} and~\ref{small}.
Then, by Theorem \ref{T04}, there exists a solution $(v^{(n+1)}, q^{(n+1)})$ to the system~\eqref{EQ250}, which satisfies the inequality
   \begin{align}
  \begin{split}
   &
   \Vert v^{(n+1)}(t) \Vert_{H^{3.5}}
   + \Vert \nabla q^{(n+1)}(t) \Vert_{H^{1.5}} 
   \\&\indeq
   \lec
   \Vert v_{0} \Vert_{H^{3.5}} 
    + \int_{0}^{t} P(
                      \Vert w^{(n)}(s) \Vert_{H^{5}(\Gamma_{1})},
                      \Vert w^{(n)}_{t}(s) \Vert_{H^{3}(\Gamma_{1})}
        )\, ds
    \\&\indeq
    \lec
    \Vert v_{0} \Vert_{H^{3.5}} 
    + \int_{0}^{t} P(M)\, ds
  \end{split}
  \label{EQ202}
  \end{align}
on a time interval $[0,T_{1}]$, where $T_{1}$ depends only on $M$, and with $q$ determined up to a constant. We may then apply Theorem~\ref{T01} to solve \eqref{EQ253} given $q^{(n+1)}$ after adjusting it by a constant so that $\int_{\Gamma_{1}} q^{(n+1)}=0$. We obtain the solution $(w^{(n+1)}, w_{t}^{(n+1)})$ satisfying the estimate
\begin{align}
  \begin{split}
  &
 \sum_{j=0}^{2} \Vert \partial_{t}^{j} w^{(n+1)}(t) \Vert_{H^{5-2j}(\Gamma_{1})}  + \nu \left( \int_{0}^{t} \Vert w^{(n+1)}_{t}(s) \Vert^{2}_{H^{4}(\Gamma_{1})} \, ds\right )^{1/2}  
   \lesssim  \Vert w_{1} \Vert_{H^{3} (\Gamma_{1})} 
  \end{split}
   \label{EQ41}
 \end{align}
 on $[0,T_{2}]$. Here, $T_2$ depends on $M$, $\nu$, and the pressure. Next, using the Poincar\'e inequality,
 followed by the trace inequality, we see that
 \begin{align}
 \begin{split}
  \int_{0}^{T_0} \Vert q^{(n+1)} \Vert^{2}_{H^{2}(\Gamma_{1})} \, ds & \lesssim  \sum_{i=1}^{2} \int_{0}^{T_0} \Vert \partial_{x_{i}} q^{(n+1)} \Vert^{2}_{H^{1}(\Gamma_{1})} \, ds   										\lesssim \int_{0}^{T_0} \Vert \nabla q^{(n+1)} \Vert^{2}_{H^{1.5}} \, ds 
.
 \end{split}
   \label{EQ55}
 \end{align}
Combining \eqref{EQ02}, \eqref{EQ202}, and \eqref{EQ55}, we
see that $T_2$ depends only on $\nu$ and $M$. Let
\begin{equation}
	T_0=\min\{T_1,T_2\}.
\end{equation}
Adding \eqref{EQ41} to \eqref{EQ202}, we obtain for $t\in[0,T_0]$
 \begin{align}
  \begin{split}
  &  \Vert v^{(n+1)}(t) \Vert_{H^{3.5}}
   + \Vert \nabla q^{(n+1)}(t) \Vert_{H^{1.5}} 
   + \sum_{j=0}^{2} \Vert \partial_{t}^{j} w^{(n+1)}(t) \Vert_{H^{5-j}(\Gamma_{1})}
   \\&\indeq
   \lesssim     \Vert v_{0} \Vert_{H^{3.5}} + \Vert w_{1} \Vert_{H^{3} (\Gamma_{1})} + T_{0} P\left( M \right) .
   \end{split}
   \label{EQ66}
  \end{align}
Readjusting $T_{0}$ to $T'$ so that 
  \begin{equation}
   T' = \min\left\{ \frac{M}{2CP(M)}, T_{0} \right\}
   ,
   \llabel{EQ46}
  \end{equation}
we conclude 
\begin{align}
  \begin{split}
  &  \Vert v^{(n+1)}(t) \Vert_{H^{3.5}}
   + \Vert q^{(n+1)}(t) \Vert_{H^{2.5}} 
   + \sum_{j=0}^{2} \Vert \partial_{t}^{j} w^{(n+1)}(t) \Vert_{H^{5-j}(\Gamma_{1})}
    \\&\indeq
      \lec
     \Vert w_{1} \Vert_{H^{3} (\Gamma_{1})} 
	+
  \Vert v_{0} \Vert_{H^{3.5}} 
    + T' P(M)  \lec M
    ,
      \end{split}
  \end{align}
for $t \in [0,T']$.
This insures that the solution at every step of the iteration belongs to a ball of fixed radius $M$ and that the iterates are well defined on a uniform interval $[0,T']$ with $T'$ depending on $M$ and $\nu$ only. In other words, the solution map from the $n^{\text{th}}$ element of the sequence to the next, and the mapping takes the ball of radius $M$ in the space $L^{\infty}([0,T'];H^{3.5}(\Omega)\times H^{2.5}(\Omega) ) $ into itself.

To apply the Banach fixed point theorem, we also need to establish that the solution map is contractive. 
This is obtained by considering the difference of two consecutive solutions which we define by
\begin{align}
\begin{split}
 &
 ( V^{(n+1)}, Q^{(n+1)}, Z^{(n+1)}, \Theta^{(n+1)}, \Psi^{n+1}, W^{(n+1)}, G^{(n+1)} )
  \\&\indeq
  =
  (
  v^{(n+1)} - v^{(n)} ,
   q^{(n+1)} - q^{(n)} ,
   \zeta^{(n+1)} - \zeta^{(n)},
   \\&\indeq\indeq\indeq\indeq\indeq\indeq\indeq
   \theta^{(n+1)} - \theta^{(n)} ,
  \psi^{(n+1)} - \psi^{(n)}, w^{(n+1)} - w^{(n)}, E^{(n+1)} - E^{(n)} )
  .
\end{split}
   \llabel{EQ67}
\end{align}
Observe that
the difference of vorticities
satisfies the equation obtained by subtracting the two  vorticity equations of the form \eqref{EQ98} satisfied by $\theta^{(n+1)}$ and $\theta^{(n)}$
and which reads
  \begin{align}
  \begin{split}
     &
    \partial_{t} \Theta^{(n+1)}_i
      + V^{(n+1)}_k E^{(n)}_{jk} \partial_{j} \theta^{(n+1)}_i
      + v^{(n)}_k G^{(n)}_{jk} \partial_{j} \theta^{(n+1)}_i
      + v^{(n)}_k E^{(n-1)}_{jk} \partial_{j} \Theta^{(n+1)}_i 
      \\&\indeq\indeq
       -\Psi^{(n)}_t E^{(n)}_{33} \partial_{3} \theta^{(n+1)}_i  
       -\psi^{(n-1)}_t G^{(n)}_{33} \partial_{3} \theta^{(n+1)}_i  
       -\psi^{(n-1)}_t E^{(n-1)}_{33} \partial_{3} \Theta^{(n+1)}_i  
     \\&\indeq\indeq 
       - \Theta^{(n+1)}_k E^{(n)}_{jk} \partial_{j} v^{(n+1)}_i
       - \theta^{(n)}_k G^{(n)}_{jk} \partial_{j} v^{(n+1)}_i
       - \theta^{(n)}_k E^{(n-1)}_{jk} \partial_{j} V^{(n+1)}_i
    =0 
    \inon{in $\Omega$}
    .
    \end{split}
  \llabel{EQ266}
\end{align}
The estimates are performed in the lower topology on solutions to this system, and we obtain the inequality
\begin{align}
   \Vert V^{(n+1)} \Vert_{H^{2.5}}
   \leq
   P(M)
   \bigl(
    \Vert W^{(n)} \Vert_{H^{4}(\Gamma_1)}
   + \Vert W_{t}^{(n)} \Vert_{H^{2}(\Gamma_{1})}
   \bigr)
   .
   \label{EQ274}
  \end{align}
For the pressure term $Q^{(n+1)}$, we utilize that it solves an elliptic boundary value problem 
with the Neumann type boundary conditions, and we have the elliptic estimate 
 \begin{align}
   \begin{split}
   \Vert \nabla Q^{(n+1)} \Vert_{H^{1.5}}
   \lec
   \Vert V^{(n+1)} \Vert_{H^{2.5}}
   + \Vert W_{tt}^{(n)} \Vert_{L^{2}(\Gamma_1)}
   + \Vert W^{(n)}_t \Vert_{H^{2}(\Gamma_1)}
   + \Vert W^{(n)}\Vert_{H^{4}(\Gamma_1)}
   ,
   \end{split}
   \label{EQ275}
  \end{align}
where the constant depends on~$M$.
The full norm of $Q$ can then be recovered after an adjustment by an appropriate constant via a Poincar\'e type inequality 
$\Vert Q^{(n+1)}\Vert_{L^{2}}   \lec
   \Vert \nabla Q^{(n+1)}\Vert_{L^2}
$
so that $\Vert \nabla Q^{(n+1)} \Vert_{H^{1.5}}$ can be replaced by $\Vert Q^{(n+1)} \Vert_{H^{2.5}}$ in~\eqref{EQ275}.

The energy estimate for the plate equation 
yields
  \begin{align}
   \begin{split}
    &
    \Vert W^{(n+1)}(t) \Vert^{2}_{H^{4}(\Gamma_1)}
    + \Vert W_{t}^{(n+1)}(t) \Vert^{2}_{H^{2}(\Gamma_1)}
    + \Vert W_{tt}^{(n+1)}(t) \Vert^{2}_{L^{2}(\Gamma_1)}
    + \nu \int_{0}^{t} \Vert W_{t}^{(n+1)}(s) \Vert^{2}_{H^{3}(\Gamma_1)}\,ds
   \\&\indeq
    \lec \int_{0}^{t} \Vert Q^{(n+1)}(s) \Vert^{2}_{H^{1}(\Gamma_{1})} \, ds
    \lec \int_{0}^{t} \Vert Q^{(n+1)}(s) \Vert^{2}_{H^{1.5}(\Omega)} \, ds
     \lec 
     T '\Vert Q^{(n+1)} \Vert^{2}_{L^{\infty}([0,T];H^{1.5}(\Omega))}
   ,
   \end{split}
   \label{EQ278}
  \end{align}
for $t\in[0,T']$.
Using \eqref{EQ274} in  \eqref{EQ275} and then  \eqref{EQ275} in \eqref{EQ278}, we get
  \begin{align}
   \begin{split}
     &\sum_{j=0}^{2}
     \Vert \partial_{t}^{j} W^{(n+1)}(t) \Vert^{2}_{L^{\infty}H^{4-2j}((0,T') \times \Gamma_1)}
         + \nu \Vert W_{t}^{(n+1)} \Vert^{2}_{L^{2}H^{3}((0,T') \times \Gamma_1)}
     \\&\indeq
     \lec 
     T' 
     P(M)
      \sum_{j=0}^{2} \Vert \partial_{t}^{j} W^{(n)}(t) \Vert^{2}_{L^{\infty}H^{4-2j}((0,T')\times\Gamma_1)}
       .
   \end{split}
   \label{EQ280}
  \end{align}
Letting $T'$ be sufficiently small, we obtain the contractive property for~$W^{n}$. Similarly, multiplying
\eqref{EQ275} by a small constant and adding to \eqref{EQ274}, we obtain an analogous estimate for $V^{n+1}$ and $Q^{(n+1)}$ from which the contraction property is established once again by taking $T'$ sufficiently small.
By the Banach fixed point theorem, there exists a solution $(v,q,w) \in L^{\infty}([0,T']; H^{3.5} (\Omega) \times H^{2.5}(\Omega))  \times L^{\infty}([0,T'); H^{5}(\Gamma_{1}))$  
 to the system on a time interval $T'>0$ depending on $M$ and $\nu>0$ and the solution is unique in the ball of radius~$M$.
\end{proof}

\subsection{Passage to the limit as viscosity goes to zero}

To obtain solutions for the system without damping ($\nu =0$), we utilize the a~priori estimates which are independent of the parameter $\nu>0$. In particular, applying the a priori estimates to the constructed solution $(v^{(m)}, q^{(m)}, w^{(m)})$ given $\nu_{m}= \fractext{1}{m}$, we can derive uniform bounds on the solution, under which we can then obtain appropriate weakly and weakly-* convergent sequences allowing us to pass to the limit as $m \to \infty$ in the system and then finally obtain solutions of the system for $\nu=0$. We omit further details since they involve standard methods, which are similar to the scheme in~\cite{KT1}. We can therefore extend the local existence Theorem~\ref{T2} to the system when~$\nu=0$.

\section*{Acknowledgments}
IK was supported in part by the
NSF grant
DMS-2205493, while NM was supported by the Croatian Science Foundation under project number IP-2022-10-2962.

\end{document}